\documentclass[a4paper,11pt,twoside]{amsart}

\usepackage{amsmath}
\usepackage{amsthm}
\usepackage{amssymb}
\usepackage[all]{xy}
\usepackage{hyperref}
\usepackage{rotating}

\newtheorem{thm}{Theorem}[section]
\newtheorem{prop}[thm]{Proposition}
\newtheorem{lem}[thm]{Lemma}
\newtheorem{cor}[thm]{Corollary}

\newtheorem{qn}[thm]{Question}

\numberwithin{equation}{section}

\theoremstyle{definition}

\newtheorem{remark}[thm]{Remark}

\newcommand{\Db}{{\rm D}^{\rm b}}

\newcommand{\Br}{{\rm Br}}
\newcommand{\NS}{{\rm NS}}
\newcommand{\Pic}{{\rm Pic}}
\newcommand{\ch}{{\rm ch}}

\newcommand{\rk}{{\rm rk}}
\newcommand{\coh}{{\cat{Coh}}}

\newcommand{\Hom}{{\rm Hom}}

\renewcommand{\Re}{{\rm Re}}
\renewcommand{\Im}{{\rm Im}}

\newcommand{\co}{{\mathrm{H}}}

\newcommand{\id}{{\rm id}}

\newcommand{\comp}{\circ}

\newcommand{\dual}{^{\vee}}
\newcommand{\mono}{\hookrightarrow}
\newcommand{\epi}{\twoheadrightarrow}
\newcommand{\mor}[1][]{\xrightarrow{#1}}
\newcommand{\isomor}{\mor[\sim]}

\newcommand{\cat}[1]{\begin{bf}#1\end{bf}}

\newcommand{\HH}{\mathbb{H}}

\newcommand{\cal}{\mathcal}
\newcommand{\ka}{{\cal A}}
\newcommand{\kb}{{\cal B}}
\newcommand{\kc}{{\cal C}}
\newcommand{\kd}{{\cal D}}
\newcommand{\ke}{{\cal E}}
\newcommand{\kf}{{\cal F}}
\newcommand{\kg}{{\cal G}}
\newcommand{\kh}{{\cal H}}
\newcommand{\kk}{{\cal K}}

\newcommand{\kj}{{\cal J}}
\newcommand{\kl}{{\cal L}}

\newcommand{\ko}{{\cal O}}
\newcommand{\kp}{{\cal P}}

\newcommand{\ks}{{\cal S}}

\newcommand{\ky}{{\cal Y}}

\newcommand{\ZZ}{\mathbb{Z}}
\newcommand{\QQ}{\mathbb{Q}}
\newcommand{\RR}{\mathbb{R}}
\newcommand{\CC}{\mathbb{C}}

\newcommand{\PP}{\mathbb{P}}

\newcommand{\h}{l}

\newcommand{\e}{\ke_0}
\newcommand{\f}{\ke_1}

\newcommand{\mc}[1]{{\cal #1}}

\begin{document}

\title{Fano varieties of cubic fourfolds containing a plane}

\author[E.\ Macr\`i and P.\ Stellari]{Emanuele Macr\`i and Paolo Stellari}

\address{E.M.: Department of Mathematics, University of Utah, 155 South 1400 East, Salt Lake City, UT 84112, USA \& Mathematical Institute, University of Bonn, Endenicher Allee 60, D-53115 Bonn, Germany}
\curraddr{Department of Mathematics, The Ohio State University, 231 W 18th Avenue, Columbus, OH 43210, USA}
\email{macri.6@math.osu.edu}

\address{P.S.: Dipartimento di Matematica ``F. Enriques'',
Universit{\`a} degli Studi di Milano, Via Cesare Saldini 50, 20133
Milano, Italy} \email{paolo.stellari@unimi.it}

\keywords{Derived categories, cubic fourfolds}

\subjclass[2010]{18E30, 14E08}

\begin{abstract}
We prove that the Fano variety of lines of a generic cubic fourfold containing a plane is isomorphic to a moduli space of twisted stable complexes on a K3 surface. On the other hand, we show that the Fano varieties are always birational to moduli spaces of twisted stable coherent sheaves on a K3 surface. The moduli spaces of complexes and of sheaves are related by wall-crossing in the derived category of twisted sheaves on the corresponding K3 surface.
\end{abstract}

\maketitle

\section{Introduction}\label{sec:intro}

In this paper we investigate the geometry of the Fano variety of lines on a cubic fourfold containing a plane and its relation to moduli spaces of sheaves on K3 surfaces.
Our approach is based on techniques arising from recent work by Kuznetsov on semiorthogonal decompositions of the derived category of coherent sheaves.

\smallskip

A \emph{cubic fourfold} is a smooth complex hypersurface of degree $3$ in $\PP^5$.
The moduli space $\kc$ of cubic fourfolds is a quasi-projective variety of dimension $20$ (\cite{H2}).
A way to produce divisors in $\kc$ is to consider \emph{special} cubic fourfolds, namely cubic fourfolds containing an algebraic surface not homologous to a complete intersection. More precisely, a cubic fourfold $Y$ is special if $\co^4(Y,\ZZ)$ contains a primitive rank-$2$ sublattice $L$ generated by $H^2$, the self-intersection of the hyperplane section, and an algebraic surface $T$. If $d$ is a positive integer, $Y$ is contained in the divisor $\kc_d\subseteq\kc$ if there exists $L$ as above and the matrix representing the intersection form on $L$ has determinant $d$. Under some assumptions on $d$, $\kc_d$ is non-empty and irreducible (see \cite{H2}). The basic example we deal with in this paper is $\kc_8$, the divisor parametrizing cubic fourfolds containing a plane (\cite{H2,V}).

\smallskip

The \emph{Fano variety of lines} $F(Y)$ of a cubic fourfold $Y$ is the variety parametrizing the lines in $Y$.
It is a projective irreducible symplectic complex manifold of dimension $4$.
A classical result of Beauville and Donagi \cite{BD} says that $F(Y)$ is deformation equivalent to the Hilbert scheme $\mathrm{Hilb^2}(S)$ of length-$2$ $0$-dimensional subschemes of a K3 surface $S$ of degree $14$. In other words, $F(Y)$ can be deformed to a moduli space of stable sheaves on $S$.

As the moduli space of smooth projective K3 surfaces is a countable union of $19$-dimensional varieties, $F(Y)$ cannot always be isomorphic to a moduli space of stable sheaves on a K3 surface. Nevertheless, the following result was proved in \cite{H2}:

\begin{thm}\label{thm:Hassett} {\bf (Hassett)}
	Assume that $d=2(n^2+n+2)$ where $n$ is an integer $\geq 2$ and let $Y$ be a generic cubic fourfold in $\kc_d$. Then $F(Y)$ is isomorphic to the Hilbert scheme of length-$2$ $0$-dimensional subschemes of a K3 surface and so, in particular, to a moduli space of stable sheaves.
\end{thm}

The case $d=8$, corresponding to cubic fourfolds $Y$ containing a plane $P$, is not covered by this result. This happens despite the fact that a very natural K3 surface is related to the geometry of $Y$. Indeed, consider the projection $Y\dashrightarrow\PP^2$ from $P$ onto a plane in $\PP^5$ disjoint from $P$. The blow up of $Y$ along $P$ yields a quadric fibration $\pi:\widetilde Y\to\PP^2$ whose fibres degenerate along a plane sextic $C\subseteq\PP^2$. Assume that $C$ is smooth (which is generically the case). The double cover $S\to\PP^2$ ramified along $C$ is a smooth projective K3 surface.

This construction provides also a natural element $\beta$ in the Brauer group of $S$ which is either trivial or of order $2$.
The geometric explanation for $\beta$ goes essentially back to \cite{V}. Indeed, the K3 surface $S$ can be thought of as the `moduli space of rulings' of the quadrics in the fibration $\pi:\widetilde Y\to\PP^2$.
Hence, this provides a $\PP^1$-fibration $F\to S$ parametrizing the lines $l\subseteq Y$ contained in the fibres of $\pi$. It turns out that $F$ is actually the Brauer--Severi variety corresponding to the twist $\beta$ (see Section \ref{subsec:twistedK3} for more details).
It is worth mentioning that, by construction, the twist $\beta$ is trivial for all rational cubic fourfolds contained in the divisors of $\kc_8$ described in \cite{Ha1}.

\smallskip

As many interesting examples of cubic fourfolds with a rich geometry are not considered in Hassett's result, one may wonder whether Theorem \ref{thm:Hassett} can be extended further. More precisely, following \cite{H2} and \cite{KMM}, one can raise the following natural question.

\begin{qn}\label{qn:main}
Are there other values of $d$ for which a (generic) cubic fourfold $Y$ has Fano variety $F(Y)$ isomorphic to a moduli space of stable (twisted) sheaves or complexes on a K3 surface?
\end{qn}

This paper may be considered as an attempt to answer this question in the case of cubic fourfolds containing a plane. The interest of cubic fourfolds of this special type is mostly related to their rich but rather mysterious geometry. For example, on one hand all new examples of rational cubic fourfolds described in \cite{Ha1} are contained in $\kc_8$ while, on the other hand, the very generic element in this divisor is expected to be non-rational.

Kuznetsov recently proposed in \cite{Kuz3} a conjectural interpretation of the rationality problem for cubic fourfolds in terms of the non-trivial part $\cat{T}_Y$ of a semi-orthogonal decomposition of the bounded derived category of coherent sheaves on $Y$ (see Section \ref{subsec:semiorthogonal} for more details about the categorical setting). Although at the moment it is not completely understood how much this approach can overcome the classical one via Hodge theory, it is certainly clear that Kuznetsov's idea sheds light on the possible use of the subcategory $\cat{T}_Y$ for geometric purposes.

For smooth cubic threefolds $Y$, this approach has been successfully investigated in \cite{BMMS} where we prove that the category $\cat{T}_Y$ characterizes uniquely the isomorphism type of $Y$. In the same paper, some further results concerning cubic fourfolds containing a plane are discussed (see \cite{BMMS}). This paper follows precisely this direction and studies further the meaning of the category $\cat{T}_Y$ in this geometric context.

\medskip

Therefore, going back to the problem of describing the geometry of the Fano variety, let $Y$ be a cubic fourfold containing a plane $P$. The fact that, in general, $F(Y)$ is not isomorphic (and not even birational) to a smooth projective moduli space of untwisted sheaves can be easily proved, even for $Y$ containing a plane (see, for example, Proposition \ref{prop:nospmod}). Nevertheless, remembering the definition of the K3 surface $S$ above and passing to $\beta$-twisted sheaves and complexes (see Section \ref{subsec:twistedK3} for the precise definitions), we can state the first main result of the paper.

\begin{thm}\label{thm:main1}
If $Y$ is a generic cubic fourfold containing a plane, then $F(Y)$ is isomorphic to a moduli space of stable objects in the derived category $\Db(S,\beta)$ of bounded complexes of $\beta$-twisted coherent sheaves on $S$.
\end{thm}

In the statement (an later on), a cubic fourfold $Y$ with a plane is meant to be \emph{generic} if it is in the complement of a countable union of codimension-$1$ subvarieties of $\kc_8$. Examples of cubic fourfolds for which Theorem \ref{thm:main1} holds are those that in Section \ref{subsec:bir1} we call \emph{very generic}. Roughly speaking, a cubic fourfold $Y$ with a plane is very generic if the algebraic part of $\co^4(Y,\ZZ)$ is the smallest possible. As these cubic fourfolds are dense in $\kc_8$, our result can be regarded as the analogue of Theorem \ref{thm:Hassett}.

The proof of Theorem \ref{thm:main1} will be carried out in Section \ref{sec:iso}. We will show that the result could be stated in a more precise but less compact way. Indeed, the Mukai vector of the complexes parametrized by the moduli space mentioned in the statement above and the Bridgeland's stability conditions for which they are stable can be explicitly described.

\medskip

The other direction we follow in the paper consists in weakening Question \ref{qn:main}. Indeed, one can wonder if the Fano variety of lines of a cubic fourfold is birational to a smooth projective moduli space of twisted sheaves on a K3 surface. The same dimension counting as at the beginning of the introduction shows that this cannot be true for all cubic fourfolds in $\kc$.

One geometric condition that may give a hope for a positive answer to the previous question is the presence of a K3 surface associated to the geometry of the fourfold. This guess is confirmed by the second main result of the paper.

\begin{thm}\label{thm:main2}
For all cubic fourfolds $Y$ containing a plane, the Fano variety $F(Y)$ is birational to a smooth projective moduli space of twisted sheaves on a K3 surface. Moreover, if $Y$ is very generic, then such a birational map is either an isomorphism or a Mukai flop.
\end{thm}

The plane sextic $C$ is again the one presented at the beginning of the introduction while the K3 surface mentioned in the statement is a special and explicit deformation of the double cover of $\PP^2$ introduced above. Also in this case, the Mukai vector of the stable twisted sheaves can be explicitly described. Moreover, the Hodge (and lattice) structure of the second cohomology of the moduli space will be made apparent in the course of the proof explained in Sections \ref{subsec:bir1} and \ref{subsec:bir2}.

\medskip

Let us spend a few words to clarify the relation between the results above.
The Fano scheme of lines on a cubic fourfold $Y$ containing a plane $P$ gives an explicit family of objects in the derived category $\Db(S,\beta)$.
A generic object in this family is a (twisted) sheaf, but some objects are complexes (namely those corresponding to the lines contained in $P$). So, there are two ways
to deal with this family. The first one consists in finding an appropriate $t$-structure and Bridgeland's stability conditions such that the base of this family (i.e., $F(Y)$) can be considered as a moduli space of stable complexes.
The other way is to consider the moduli space of sheaves, containing the generic
object of the family. This gives a moduli space of sheaves, which by construction is
birational to $F(Y)$. Moreover, for generic $Y$ the birational transformation is shown
to be a single Mukai flop in the plane $P\dual\subseteq F(Y)$.

Notice that the moduli space of complexes in Theorem \ref{thm:main1} can be regarded as a compactification of an open subset of the moduli space of twisted stable sheaves in Theorem \ref{thm:main2} via complexes. The two compactifications are related by a wall-crossing phenomenon in the derived category. In geometric terms, as observed in Remark \ref{rmk:MukaiFlopUtah}, this means that $F(Y)$ is the Mukai flop of a moduli space of twisted sheaves, for generic cubic fourfolds containing a plane.

\smallskip

Theorem \ref{prop:Hassett} shows that the condition of $F(Y)$ being birational to a moduli space of untwisted sheaves on a K3 surface is equivalent to having an exact equivalence between $\cat{T}_Y$ and the bounded derived category of coherent sheaves on a K3 surface. Under the light of Kuznetsov's conjecture, this provides a criterion for the rationality of cubic fourfolds containing a plane in terms of the geometry of the Fano variety of lines.

\medskip

Some words about the plan of the paper are in order here. Section \ref{sec:prelim} studies the semi-orthogonal decomposition of the derived category of a cubic fourfold containing a plane defined in \cite{Kuz3}. In particular, we deal with some basic properties of the ideal sheaves of lines (and of complexes closely related to them) which will be fundamental for the paper. Contrary to the order in the exposition above, we first prove Theorem \ref{thm:main2} in Section \ref{sec:stab}. This is motivated by the fact that the argument descends directly from some easy properties studied in Section \ref{subsec:twistedK3} and which will be used to prove Theorem \ref{thm:main1} as well. In Section \ref{subsec:bir3} we characterize when the Fano variety of lines is birational to a moduli space of untwisted sheaves on a K3 surface in terms of Kuznetsov's component $\cat{T}_Y$.
Finally, in Section \ref{sec:iso}, we use Bridgeland's stability conditions to prove Theorem \ref{thm:main1}. For sake of simplicity, all derived functors will be denoted as if they were underived, e.g.\ for a morphism of varieties $f:X\to Z$, we will denote $f^*$ for the derived pull-back, $f_*$ for the derived push-forward, and so on. We work over the complex numbers.

\section{The categorical setting and the ideal sheaves of lines}\label{sec:prelim}

In this section we study the semi-orthogonal decomposition of $\Db(Y)$ given in \cite{Kuz3}, for $Y$ a cubic fourfold. When $Y$ contains a plane, some basic properties of the relevant piece of this decomposition have been analyzed in \cite{BMMS}. Here we continue this study considering some special objects in it which are used in the rest of the paper.

\subsection{The geometric setting}\label{subsec:geometry}

Suppose that $Y$ is a cubic fourfold containing a plane $P$. As explained in the introduction and following \cite{Kuz3}, consider the diagram
\[
\xymatrix{D\ar[r]\ar[d]&\widetilde{Y}:={\rm Bl}_{P}Y\ar^{\sigma}[d]\ar[r]^{\;\;\;\;\;\;\;\;\pi}&\PP^2&\\
P\ar[r]&Y\subseteq\PP^5,
}
\]
where $\sigma:\widetilde{Y}\to Y$ is the blow-up of $Y$ along $P$, $D$ is the exceptional divisor, and $\pi:\widetilde{Y}\to\PP^2$ is the quadric fibration induced by the projection from $P$ onto a plane.

Denote by  $\widetilde{\PP^5}$ the blow-up of $\PP^5$ along $P$, by $h$ a hyperplane in $\PP^2$, and by $H$ (with a little abuse of notation) both a hyperplane in $\PP^5$ and its pull-backs to $\widetilde{Y}$ and to $\widetilde{\PP^5}$.
If $q:\widetilde{\PP^5}\to\PP^2$ is the induced projection from $P$, then the calculation in \cite[Lemma 4.1]{Kuz3} yields $\ko_{\widetilde{Y}}(D)\cong\ko_{\widetilde{Y}}(H-h)$ and $\widetilde{\PP}^5 \cong \PP_{\PP^2} (\ko_{\PP^2}^{\oplus 3} \oplus \ko_{\PP^2}(-h))$. Moreover, the relative ample line bundle is $\ko_{\widetilde{\PP}^5}(H)$, and the relative canonical bundle is $\ko_{\widetilde{\PP}^5}(h-4H)$.

As in \cite[Sect.\ 3]{Kuz2}, to the quadric fibration $\pi$, one can associate a sheaf $\kb_0$ (resp.\ $\kb_1$)  of even (resp.\ odd) parts of the Clifford algebra. Explicitly, as sheaves of $\ko_{\PP^2}$-modules, we have
\begin{align*}
\kb_0&\cong\ko_{\PP^2}\oplus\ko_{\PP^2}(-h)^{\oplus 3}\oplus\ko_{\PP^2}(-2h)^{\oplus 3}\oplus\ko_{\PP^2}(-3h)\\
\kb_1&\cong\ko_{\PP^2}^{\oplus 3}\oplus\ko_{\PP^2}(-h)^{\oplus 2}\oplus\ko_{\PP^2}(-2h)^{\oplus 3},
\end{align*}
where $h$ denotes, by abuse of notation, both the class of a line in $\PP^2$ and its pull-back via $\pi$.
Denote by $\coh(\PP^2,\kb_0)$ the abelian category of right coherent $\kb_0$-modules and by $\Db(\PP^2,\kb_0)$ its bounded derived category. As observed in \cite{Kuz3}, the Serre functor of $\Db(\PP^2,\kb_0)$ is the shift by $2$. A category with this property is sometimes called \emph{$2$-Calabi--Yau category}.

\begin{remark}\label{rmk:spherical}
Notice that $\kb_i$ is spherical in the category $\Db(\PP^2,\kb_0)$, for $i=0,1$. Recall that an object $\ka$ in a $\CC$-linear $2$-Calabi--Yau category $\cat{C}$ is \emph{spherical} if $\Hom_{\cat{C}}(\ka,\ka)\cong\Hom_{\cat{C}}(\ka,\ka[2])\cong\CC$ and $\Hom_{\cat{C}}(\ka,\ka[j])=0$, when $j\neq 0,2$.

To prove the claim observe that, by adjunction, $\Hom(\kb_0,\kb_0[j])\cong \co^j(\PP^2,\ko_{\PP^2}\oplus\ko_{\PP^2}(-h)^{\oplus 3}\oplus\ko_{\PP^2}(-2h)^{\oplus 3}\oplus\ko_{\PP^2}(-3h))$ and the functor $(-)\otimes_{\kb_0}\kb_1$ is an autoequivalence of $\coh(\PP^2,\kb_0)$ sending $\kb_0$ to $\kb_1$ (see \cite[Cor.\ 3.9]{Kuz2}).
\end{remark}

\medskip

Let $Y$ be a cubic fourfolds containing a plane $P$ and denote by $C$ the sextic curve in $\PP^2$ over which the fibres of $\pi$ degenerate. If $C$ is smooth, then the double cover of $\PP^2$  ramified along $C$ $$f:S\longrightarrow\PP^2$$ is a smooth K3 surface. Furthermore, by \cite[Prop.\ 3.13]{Kuz2}, there exists an Azumaya algebra $\ka_0$ on the K3 surface $S$, such that
\begin{equation}\label{eqn:Azum}
f_*\ka_0=\kb_0\qquad\mbox\qquad f_*:\coh(S,\ka_0)\isomor\coh(\PP^2,\kb_0).
\end{equation}
(See, for example, \cite{Milne} for the basic properties of Azumaya algebras, but also \cite[Chapter 1]{C}.)

\begin{remark}\label{rmk:stupid2}
An easy observation shows that the fact that $C$ is smooth coincides with the assumption in \cite{Kuz3}. More precisely, the cubic fourfold $Y$ with a plane $P$ has smooth plane sextic $C$ if and only if the fibres of $\pi:\widetilde{Y}\to\PP^2$ have at most one singular point.	Indeed, by \cite[Prop.\ 1.2]{Bea}, the curve $C$ can have at most ordinary double points. This is because we are assuming that $Y$ (and thus $\widetilde{Y}$) is smooth. Moreover a fibre of $\pi$ is union of two distinct planes if and only if it is projected to a singular point of $C$.
\end{remark}

Denote by $h$ a hyperplane in $\PP^2$. We will often need to restrict ourselves to the case where $Y$ satisfies the following additional requirement:
\begin{itemize}
\item[$(\ast)$] The cubic fourfold $Y$ contains a plane $P$, the sextic $C$ is smooth, and $f^*h$ is indecomposable.
\end{itemize}
The divisor $f^*h$ is \emph{idecomposable} if $\ko_S(f^*h)$ is not linearly equivalent to $\ko_S(C_1+C_2)$, where $C_i$ is an effective (non-trivial) divisor on $S$, for $i=1,2$.

\begin{remark}\label{rmk:stupid3}
The easiest case in which $f^*h$ is indecomposable is when $S$ is generic, i.e., $\Pic(S)=\ZZ\, f^*h$. These K3 surfaces are generic in the moduli space of degree-$2$ K3 surfaces and they are the main examples to have in mind throughout this paper.
\end{remark}

\subsection{The semi-orthogonal decomposition}\label{subsec:semiorthogonal}

Let $Y$ be a cubic fourfold and let $\Db(Y)$ be its bounded derived category of coherent sheaves.
Define $\ko_Y(iH):=\ko_{\PP^5}(iH)|_{Y}$ and
\begin{equation*}
\begin{split}
\cat{T}_Y&:=\langle\ko_Y,\ko_Y(H),\ko_Y(2H)\rangle^\perp\\
&=\left\{\ka\in\Db(Y):\Hom_{\Db(Y)}(\ko_Y(iH),\ka[p])=0,\text{ for all }p\text{ and }i=0,1,2\right\}.
\end{split}
\end{equation*}
Since the collection $\{\ko_Y,\ko_Y(H),\ko_Y(2H)\}$ is exceptional in $\Db(Y)$, we have a semi-orthogonal decomposition of $\Db(Y)$ as
\[
\Db(Y)=\langle\cat{T}_Y,\ko_Y,\ko_Y(H),\ko_Y(2H)\rangle.
\]
(For details about semi-orthogonal decompositions and exceptional collections, see, for example, \cite[Sect.\ 2.1]{Kuz3}.)

By \cite[Sect.\ 4]{Kuz2}, we can define a fully faithful functor $\Phi:=\Phi_{\ke'}:\Db(\PP^2,\kb_0)\to\Db(\widetilde{Y})$, $$\Phi_{\ke'}(\ka):=\pi^*\ka\otimes_{\pi^*\kb_0}\ke',$$ for all $\ka\in\Db(\PP^2,\kb_0)$, where $\ke'\in\coh(\widetilde{Y})$ is a rank-$4$ vector bundle on $\widetilde{Y}$ with a natural structure of flat left $\pi^*\kb_0$-module.
We will not need the actual definition of $\ke'$ (for which the reader is referred to \cite[Sect.\ 4]{Kuz2}) but only the presentation
\[
0\to q^*\kb_0(-2H)\to q^*\kb_1(-H)\to\alpha_*\ke'\to0,
\]
where $\alpha:\widetilde{Y}\mono\widetilde{\PP^5}$ is the natural embedding and $q:\widetilde{\PP^5}\to\PP^2$ is the induced projection as in Section \ref{subsec:geometry}.

\medskip

The results of \cite{Kuz2} and \cite{Orlov} give, respectively, two semi-orthogonal decompositions:
\begin{equation}\label{eqn:semiorth}
\begin{array}{l}
	\Db(\widetilde Y)=\langle\Phi(\Db(\PP^2,\kb_0)),\ko_{\widetilde Y}(-h),\ko_{\widetilde Y},\ko_{\widetilde Y}(h),\ko_{\widetilde Y}(H),\ko_{\widetilde Y}(h+H),\ko_{\widetilde Y}(2h+H)\rangle\\
\Db(\widetilde Y)=\langle\sigma^*(\cat{T}_Y),\ko_{\widetilde Y},\ko_{\widetilde Y}(H),\ko_{\widetilde Y}(2H),i_*\ko_{D},i_*\ko_{D}(H),i_*\ko_{D}(2H)\rangle.	
\end{array}
\end{equation}
Theorem 4.3 in \cite{Kuz3} yields an equivalence $$\Db(\PP^2,\kb_0)\cong\cat{T}_Y.$$ The way this is achieved is performing a sequence of mutations which allow Kuznetsov to compare the two semi-orthogonal decompositions in \eqref{eqn:semiorth}. The details will not be needed in the rest of this paper but we just recall the two fundamental mutations in the construction of the equivalence:
\begin{itemize}
	\item[(A)] The right mutation of $\Phi(\Db(\PP^2,\kb_0))$ through $\ko_{\widetilde Y}(-h)$ leading to the semi-orthogonal decomposition
	\[
	\qquad\Db(\widetilde Y)=\langle\ko_{\widetilde Y}(-h),\Phi'(\Db(\PP^2,\kb_0)),\ko_{\widetilde Y},\ko_{\widetilde Y},\ko_{\widetilde Y}(h),\ko_{\widetilde Y}(H),\ko_{\widetilde Y}(h+H),\ko_{\widetilde Y}(2h+H)\rangle,
	\]
	where $\Phi'$ is the functor obtained composing $\Phi$ with the right mutation. Namely,
	\[
	\Phi':=\mathsf{R}_{\ko_{\widetilde{Y}}(-h)}\circ\Phi:\Db(\PP^2,\kb_0)\longrightarrow\Db(\widetilde{Y}).
	\]
	\item[(B)] The left mutation of $\Phi'(\Db(\PP^2,\kb_0))$ through $\ko_{\widetilde Y}(h-H)$. Composing $\Phi'$ with this left mutation, yields a functor $\Phi'':=\mathsf{L}_{\ko_{\widetilde{Y}}(h-H)}\circ\Phi':\Db(\PP^2,\kb_0)\to\Db(\widetilde{Y})$.
\end{itemize}
As explained in the proof of \cite[Thm.\ 4.3]{Kuz3}, the equivalence $\Db(\PP^2,\kb_0)\cong\cat{T}_Y$ is provided by the functor $\sigma_*\circ\Phi''$. (For details and basic results about mutations, see, for example, \cite[Sect.\ 2.2]{Kuz3}.)

\medskip

Recall that the left adjoint functor of $\Phi$ is
\begin{align*}
\Psi(-)&:=\pi_*((-)\otimes\ko_{\widetilde{Y}}(h)\otimes\ke[1]),
\end{align*}
where $\ke\in\coh(\widetilde{Y})$ is another rank-$4$ vector bundle on $\widetilde{Y}$ with a natural structure of right $\pi^*\kb_0$-module (see \cite[Sect.\ 4]{Kuz2}). The main property we will need is the presentation
\begin{equation}\label{eq:defofke}
0\to q^*\kb_1(-h-2H)\to q^*\kb_0(-H)\to\alpha_*\ke\to0.
\end{equation}

The following result will be used to control the compatibilities with the mutations (A) and (B) above.

\begin{lem}\label{lem:stupid}
	{\rm (i)} For any integer $m$ and for $a=0,-1$, we have $\Psi(\ko_{\widetilde Y}(mh+aH))=0$.
	
	{\rm (ii)} We have
		\begin{align*}
		&\Psi(\ko_{\widetilde Y}(-h+H)) \cong\kb_0[1]\\
		&\Psi(\ko_{\widetilde Y}(h-2H)) \cong\kb_1[-1].
		\end{align*}
\end{lem}

\begin{proof}
	By the projection formula and the fact that $\pi = q \circ \alpha$, we have
	\begin{equation*}
	\begin{split}
	\Psi(\ko_{\widetilde{Y}}(mh+aH))&\cong\Psi(\alpha^*(\ko_{\widetilde{\PP^5}}(mh+aH)))\\
	               &\cong q_*(\ko_{\widetilde{\PP^5}}((m+1)h+aH)\otimes\alpha_*\ke[1]),
	\end{split}
	\end{equation*}
	for all $m$ and all $a$.
	Furthermore, since $\widetilde{\PP^5}\to\PP^2$ is a projective bundle and $\ko_{\widetilde{\PP^5}/\PP^2}(1)\cong\ko_{\widetilde{\PP^5}}(H)$, then $q_*(\ko_{\widetilde{\PP^5}}(bH))=0$, for $b=-1,-2,-3$.
	
	Hence, to prove (i), it is enough to unravel \eqref{eq:defofke} and use again the projection formula.
	The proof of (ii) is similar, by using relative Grothendieck-Serre duality and the fact that the relative canonical bundle is $\ko_{\widetilde{\PP}^5}(h-4H)$.
\end{proof}

\subsection{Ideal sheaves of lines and their relatives}\label{subsec:idealshlines}

Let $l$ be a line in $Y$ and let $I_l$ be its ideal sheaf. As remarked in \cite{KM}, while $I_l\not\in\cat{T}_Y$, the (stable) sheaf $\kf_l\in\coh(Y)$ defined by the short exact sequence
\begin{equation}\label{eqn:defFl}
0\to\kf_l\to\ko_{Y}^{\oplus 4}\to I_l(H)\to 0
\end{equation}
is in $\cat{T}_Y$. This means that, given the semi-orthogonal decomposition
\[
\Db(Y)=\langle\cat{T}'_Y,\ko_Y(-H),\ko_Y,\ko_Y(H)\rangle,
\]
the sheaf $\kf_l(-H)$ is in $\cat{T}'_Y$. Right-mutating $\cat{T}'_Y$ with respect to $\ko_Y(-H)$, one gets the semi-orthogonal decompositions
\[
\Db(Y)=\langle\ko_Y(-H),\cat{T}''_Y,\ko_Y,\ko_Y(H)\rangle=\langle\cat{T}''_Y,\ko_Y,\ko_Y(H),\ko_Y(2H)\rangle.
\]
In particular, the right mutation
\begin{equation}\label{eqn:defPl}
P_l:=\mathrm{cone}(ev\dual:\kf_l(-H)\to\mathrm{R}\Hom(\kf_l(-H),\ko_Y(-H))\dual\otimes\ko_Y(-H))[-1]
\end{equation}
is contained in $\cat{T}''_Y$. Being $\cat{T}_Y$ and $\cat{T}''_Y$ right orthogonals in $\Db(Y)$ of the same full subcategory, they can be naturally identified and so $P_l$ is an object of $\cat{T}_Y$.
Notice that we have an exact triangle
\[
\ko_Y(-H)[1]\longrightarrow P_l\longrightarrow I_l.
\]

\begin{remark}\label{rmk:modsp}
	Due to \cite[Prop.\ 5.5]{KM}, the Fano variety $F(Y)$ is identified with the connected component of the moduli space of stable sheaves containing the objects $\kf_l$, for any line $l$ in $Y$. Since the objects $P_l$ are obtained by tensoring $\kf_l$ and applying a right mutation (which yield an autoequivalence of $\cat{T}_Y$), $F(Y)$ can be also identified with the moduli space of the objects $P_l$.
\end{remark}

The following result which calculates the preimage
\begin{equation}\label{eq:Dan}
\kl_l:=(\sigma_*\circ\Phi'')^{-1}(P_l)[1]
\end{equation}
in $\Db(\PP^2,\kb_0)$ is particularly relevant for the proof of Theorems \ref{thm:main1} and \ref{thm:main2}.

\begin{prop}\label{prop:ideallines}
	If $l$ is not contained in $P$ then $\kl_l$ is a pure torsion sheaf supported on a line in $\PP^2$. If $l\subseteq P$, then there exists an exact triangle
	\[
	\kb_0[1]\longrightarrow\kl_l\longrightarrow\kb_1,
	\]
	where the extension map $\kb_1\to\kb_0[2]$ is uniquely determined by the line $l$. Moreover, there exists a bijective correspondence between the Fano variety $P\dual$ of lines in $P$ and the isomorphism classes of extensions $\kb_1\to\kb_0[2]$.
\end{prop}

\begin{proof}
Before going into the rather technical details of the proof, let us briefly outline the strategy. First observe that, as $\sigma_*\comp\Phi''$ is an equivalence onto its image, its quasi-inverse is given by its left adjoint functor $\Psi''\comp\sigma^*$, where $\Psi''$ is the left adjoint of $\Phi''$. Since $\Phi''$ is the composition of $\Phi$ with two mutations and the adjoint of these mutations are again appropriate mutations, we can conclude that $\Psi''$ is the composition $\Psi$ with mutations functors. Now Lemma \ref{lem:stupid} tells us that the latter functors do not effect the calculations.	
	
Let us now be more precise and give the details of the above discussion. First observe that $(\sigma_*\circ\Phi'')^{-1}(I_l)\cong\Psi(\sigma^*I_l)$.
Indeed, by Lemma \ref{lem:stupid}, the mutations (A) and (B) by $\ko_{\widetilde{Y}}(-h)$ and  $\ko_{\widetilde{Y}}(h-H)$ have no effect.
Applying $\sigma^*$ to \eqref{eqn:defFl} (tensored by $\ko_{\widetilde Y}(-H)$) and to \eqref{eqn:defPl} and observing that, again by Lemma \ref{lem:stupid},
$\Psi(\ko_{\widetilde Y}(-H))=0$, we have
\[
\kl_l=(\sigma_*\circ\Phi'')^{-1}(\kf_l(-H))[1]=(\sigma_*\circ\Phi'')^{-1}(I_l).
\]

The easiest case is when $l\cap P=\emptyset$. Since the rational map $Y\dashrightarrow\PP^2$ is well-defined on $l$ and map it to another line, denote by $j$ the embedding $l\mono\widetilde{Y}\mor[\pi]\PP^2$.
	Pulling back via $\sigma$ the exact sequence
	\[
	0\longrightarrow I_l\longrightarrow\ko_Y\longrightarrow\ko_l\longrightarrow 0,
	\]
	yields
	\[
	0\longrightarrow\sigma^*I_l\longrightarrow\ko_{\widetilde{Y}}\longrightarrow\sigma^*\ko_l=\ko_l\longrightarrow 0.
	\]
	By Lemma \ref{lem:stupid}, we have $\Psi(\ko_{\widetilde{Y}})=0$ and so
	\begin{equation*}
	\begin{split}
	\Psi(\sigma^*I_l)&\cong\Psi(\ko_l)[-1]\\
	    &=\pi_*(\ko_l[-1]\otimes\ko_{\widetilde{Y}}(h)\otimes\ke[1]))\\
	    &\cong j_*(j^*\ko_{\widetilde{Y}}(h)\otimes\ke |_l)\\
	    &\cong j_*(\ke |_l)\otimes\ko_{\PP^2}(h),
	\end{split}
	\end{equation*}
which is precisely the claim.

Assume now $l\cap P=\{pt\}$. A local computation shows that $\sigma^*I_l$ is a sheaf, sitting in an exact sequence
\[
0\longrightarrow\sigma^*I_l\longrightarrow\ko_{\widetilde{Y}}\longrightarrow\ko_{l\cup\gamma}\longrightarrow 0,
\]
where $\gamma:=\sigma^{-1}(pt)$ and $l$ denotes, by abuse of notation, the strict transform of $l$ inside $\widetilde{Y}$.
Now, by Lemma \ref{lem:stupid}, we have
\[
\kl_l\cong\Psi(\sigma^*I_l)\cong\Psi(\ko_{l\cup\gamma})[-1]\cong\pi_*(\ke|_{l\cup\gamma}\otimes\ko_{\widetilde{Y}}(h)).
\]
By using the exact sequence
\[
0\longrightarrow\ko_{\gamma}(-h)\longrightarrow\ko_{l\cup\gamma}\longrightarrow\ko_l\longrightarrow 0,
\]
and observing that $\ko_{\widetilde{Y}}(h)|_l\cong\ko_l$, we have an exact triangle
\begin{eqnarray}\label{eqn:coho}
\pi_*(\ke|_{\gamma})\longrightarrow\kl_l\longrightarrow\pi_*(\ke|_l),
\end{eqnarray}
where $\pi_*(\ke|_{\gamma})$ is a torsion sheaf supported on the line $\pi(\gamma)$ in $\PP^2$, since $\pi$ is a closed embedding on $\gamma$.
On the other side, the sheaf $\ke|_{l}$ has no higher cohomology: indeed, by \eqref{eq:defofke}, $\ke|_l$ is a quotient of $\ko_l(-1)^{\oplus 8}$.
Hence $\pi_*(\ke|_{l})$ is a torsion sheaf supported on a point, and so $\kl_l$ is a torsion sheaf supported on a line in $\PP^2$.
It is not too hard to see that $\kl_l$ is actually a pure sheaf of dimension $1$: indeed, this follows directly from $\Hom(\kl_l,\kl_l)\cong\CC$.

Finally, let us consider the case $l\subseteq P$ and denote by $C'$ the preimage of $l$ via $\sigma$. By Lemma \ref{lem:stupid}, applied to
\[
\sigma^*I_l\longrightarrow\ko_{\widetilde{Y}}\longrightarrow\sigma^*\ko_l,
\]
we know $\kl_l\cong\Psi(\sigma^*\ko_l)[-1]$. By \cite[Prop.\ 11.12]{Huy}, we have an exact triangle
\[
\ko_{C'}(D)[1]\longrightarrow\sigma^*(\ko_{l})\longrightarrow\ko_{C'},
\]
where, as before, $D$ denotes the exceptional divisor of $\widetilde{Y}$.
By the fact that $\ko_D(-C')\cong\ko_D(-H)$ and by Lemma \ref{lem:stupid} applied to the short exact sequence
\[
0\longrightarrow\ko_{\widetilde{Y}}(h-H)=\ko_{\widetilde{Y}}(-D)\longrightarrow\ko_{\widetilde{Y}}\longrightarrow\ko_D\longrightarrow 0,
\]
we get $\Psi(\ko_D)=0$. Moreover, the short exact sequence
\[
0\longrightarrow\ko_D(-H)\longrightarrow\ko_D\longrightarrow\ko_{C'}\longrightarrow 0,
\]
yields $\Psi(\ko_{C'})\cong\Psi(\ko_D(-H))[1]$. On the other hand, again Lemma \ref{lem:stupid} shows that $\Psi(\ko_D(-H))[1]\cong\Psi(\ko_{\widetilde{Y}}(h-2H))[2]\cong\kb_1[1]$.
This gives the morphism $\kl_l\to\kb_1$.
The argument to show $\Psi(\ko_{C'}(D))\cong\kb_0$ is similar and it is left to the reader.

To prove the last claim, observe that each line $l\subseteq P$ determines uniquely the (isomorphism class of the) object $P_l$ and hence the extension $\kb_1\to\kb_0[2]$ in the above constuction. This yields an injection $P\dual\hookrightarrow\PP(\Hom(\kb_1,\kb_0[2]))$. But now, by Serre duality and adjunction, $\Hom(\kb_1,\kb_0[2])\cong \Hom(\kb_0,\kb_1)\dual\cong\Hom(\ko_{\PP^2},\kb_1)\dual\cong\CC^{\oplus 3}$.
Here we used the fact that the Serre functor of $\cat{T}_Y$ is the shift by $2$.
\end{proof}

\section{Birationality and moduli spaces of sheaves}\label{sec:stab}

In this section we prove Theorem \ref{thm:main2}.
To make clear how the proof goes, let us explain in some detail the structure of this section. In Section \ref{subsec:bir1} we deal with the case of cubic fourfolds satisfying $(\ast)$. In this case the computation can be carried out explicitly and a direct computation provides a description of the birational map. Section \ref{subsec:bir2} completes the proof studying the general case. Here we use a deformation argument based on Hodge theory.
Along the way, we recall in Section \ref{subsec:twistedK3} some facts about twisted sheaves and twisted K3 surfaces.

\subsection{Twisted K3 surfaces}\label{subsec:twistedK3}

Let $X$ be a smooth projective variety and let $\Br(X)$ be its \emph{Brauer group}, i.e., the torsion part of the
cohomology group $\co^2(X,\ko_X^*)$ in the analytic topology. (For more details about Brauer groups, see \cite{Milne}.)

Recall that any $\beta\in\Br(X)$ can be represented by a \v{C}ech cocycle on
an open analytic cover $\{U_i\}_{i\in I}$ of $X$ using the
sections $\beta_{ijk}\in\Gamma(U_i\cap U_j\cap
U_k,\mathcal{O}^*_X)$. A \emph{$\beta$-twisted coherent sheaf}
$\kf$ consists of a collection $(\{\kf_i\}_{i\in
I},\{\varphi_{ij}\}_{i,j\in I})$, where $\kf_i$ is a coherent
sheaf on $U_i$ and $\varphi_{ij}:\kf_j|_{U_i\cap
U_j}\to\kf_i|_{U_i\cap U_j}$ is an isomorphism satisfying the
following conditions:
\[
\mbox{$\varphi_{ii}=\mathrm{id}$;\;\;\;\;\;
$\varphi_{ji}=\varphi_{ij}^{-1}$;\;\;\;\;\;
$\varphi_{ij}\circ\varphi_{jk}\circ\varphi_{ki}=\beta_{ijk}\cdot\mathrm{id}$.}
\]
By $\coh(X,\beta)$ we denote the abelian category of
$\beta$-twisted coherent sheaves on $X$, while $\Db(X,\beta)$
is the bounded derived category $\Db(\coh(X,\beta))$. A \emph{twisted variety} is a pair $(X,\beta)$, where $X$ is a smooth projective variety and $\beta\in\Br(X)$. (For more details about twisted sheaves, see \cite{C}.)

If now $S$ is a K3 surface, one can lift $\beta$ to a \emph{$B$-field} $B\in \co^2(S,\QQ)$. More precisely, using that $\co^3(S,\ZZ)=0$, the long exact sequence in cohomology associate to the exponential short exact sequence on $S$, allows us to lift $\beta\in \co^2(S,\ko_S^*)$ to an element $B\in \co^2(S,\QQ)$ (see, for example, \cite[Sect.\ 1]{HS1}).

For $\sigma$ a generator of $\co^{2,0}(S)$, let $\sigma_B:=\sigma+\sigma\wedge B\in \co^*(S,\CC)$ be the \emph{twisted period} of $(S,\beta)$.
Let $T(S,B)\subseteq \co^*(S,\ZZ)$ be the \emph{twisted transcendental lattice}, i.e., the minimal primitive sublattice of $\co^*(S,\ZZ)$ such that $\sigma_B\in T(S,B)\otimes\CC$. (The lattice structure is given by the Mukai pairing on $\co^*(S,\ZZ)$ as explained in \cite[Sect.\  5.2]{Huy}.) Moreover, we denote by $\Pic(S,B):=T(S,B)^\perp$ the \emph{twisted Picard lattice}, where the orthogonal is taken in $\co^*(S,\ZZ)$ with respect to the Mukai pairing. Recall that the \emph{Mukai pairing} is defined, in terms of cup product, as
\[
\langle v,w\rangle:=-v_0\cup w_4+v_2\cup w_2-v_4\cup w_0,
\]
for every $v=(v_0,v_2,v_4)$ and $w=(w_0,w_2,w_4)$ in $\co^*(S,\mathbb{Z})$.

In \cite{HS1}, it was also defined a twisted weight-$2$ Hodge structure on the total cohomology of $S$ which, together with the Mukai pairing, is then denoted by $\widetilde \co(X,B,\ZZ)$.

Using the $B$-field lift of $\beta\in\Br(S)$, the twisted Chern character ${\rm ch}^B:\coh(S,\beta)\to\widetilde \co(S,B,\ZZ)$ is defined in \cite[Sect.\ 1]{HS1}.
The \emph{Mukai vector} of $\kf\in\Db(S,\beta)$ is $v^B(\kf):=\ch^B(\kf)\cdot\sqrt{\mathrm{td}(S)}\in\widetilde \co(S,B,\ZZ)$.  For $\kf\in\coh(S,\beta)$ we denote by $\rk(\kf)$ and $c_1^B(\kf)$ the degree $0$ and $2$ parts of $\ch^B(\kf)$. Besides the presence of several definition of Chern characters (i.e., twisted or not), in the rest of the paper the Picard group $\Pic(S)$ will be always meant to be embedded in $\co^2(S,\ZZ)$ by taking the untwisted Chern class $c_1(=c_1^0)$.

\begin{lem}\label{lem:NS}
	If $\beta\in\Br(S)$ has order $d$, the lattice $\Pic(S,B)$ is generated by $\Pic(S)$ and the vectors $w_1=(d,dB,0)$ and $w_2=(0,0,1)$ in $\co^*(S,\ZZ)=\co^0(S,\ZZ)\oplus \co^2(S,\ZZ)\oplus \co^4(S,\ZZ)$.
\end{lem}

\begin{proof}
Let $c$ be the discriminant of $\Pic(S)$. Recall that the discriminant of a lattice is the determinant of the matrix representing the corresponding symmetric bilinear form  (see \cite[Sect.\ 1]{Ni} for more details about it). By \cite[Prop.\ 1.6.1]{Ni}, the discriminants of $\Pic(S)$ and $T(S)$ coincide, up to sign. For the same reason, this happens for the discriminants of $\Pic(S,B)$ and $T(S,B)$ as well. Now $T(S,B)$ can be realized as a (non-primitive) sublattice of $T(S)$ of index $d$ and so, up to sign, the discriminant of $T(S,B)$ and of $\Pic(S,B)$ is $d^2c$ (see, for example, \cite[Remark 3.1]{HS1}).

Obviously the two vectors $w_1,w_2$ and the elements of $\Pic(S)$ are orthogonal to $\sigma_B$ and thus belong to $\Pic(S,B)$. Let $L$ be the sublattice of $\Pic(S,B)$ generated by those elements. An easy calculation shows that $d^2c$ is the discriminant of $L$. Therefore $\Pic(S,B)$ and $L$ have the same discriminant and so they coincide.
\end{proof}

The following is a straightforward consequence of the result above.

\begin{cor}\label{cor:rankdiv}
	If $\beta\in\Br(S)$ has order $d$, then the rank of any $\ka\in\coh(S,\beta)$ is divisible by $d$.
\end{cor}

\smallskip

Moving to the geometric situation we are interested in, assume now that $Y$ is a cubic fourfold satisfying $(\ast)$. We have observed in Section \ref{subsec:semiorthogonal} that there exists an Azumaya algebra $\ka_0$ on $S$, the K3 surface which is the double cover of $\PP^2$ ramified along $C$, such that $\Db(S,\ka_0)\cong\Db(\PP^2,\kb_0)$ and $f_*\ka_0=\kb_0$. Such an Azumaya algebra corresponds to the choice of the element $\beta$ in the $2$-torsion part of the Brauer group $\Br(S)$ of $S$ which, in turn, corresponds to the $\PP^1$-fibration $F\to S$ parametrizing the lines contained in the fibres of $\pi:\widetilde{Y}\to\PP^2$. Therefore $\Db(S,\beta)\cong\Db(\PP^2,\kb_0)$ and this equivalence is realized by the following composition of equivalences
\begin{equation}\label{eqn:equiv}
\xymatrix{
\Xi:\coh(S,\beta)\ar[rr]^{(-)\otimes_{\ko_S}\e\dual}&&\coh(S,\ka_0)\ar[rr]^{f_*}&&\coh(\PP^2,\kb_0),
}
\end{equation}
where $\e$ is a rank-$2$ locally free $\beta$-twisted sheaf such that $\ka_0\cong\ke nd(\e)$. Obviously, $\Xi(\e)=\kb_0$. For sake of simplicity we denote by $\Xi$ the induced functor on the level of derived categories.
Set $\ke_1\in\coh(S,\beta)$ and, for $l\subseteq Y$ a line, $\kj_l\in\Db(S,\beta)$ to be such that
\[
\kb_1=\Xi(\ke_1)\qquad\mathrm{and}\qquad\kl_l=\Xi(\kj_l).
\]
Notice that $\ke_1$ is locally-free of rank 2.
Also, for a line $l\subseteq Y$ not contained in the plane $P$, the sheaf $\kj_l$ is pure of dimension $1$.
Indeed, by Proposition \ref{prop:ideallines}, $\kl_l$ is a vector bundle (of rank 4) supported on a line in $\PP^2$. (See \eqref{eq:Dan} for the definition of $\kl_l$.)

\begin{lem}\label{lem:calc}
	Let $Y$ be a cubic fourfold satisfying $(\ast)$. If $l$ is a line in $Y$, then $v^B(\kj_l)=(0,f^*h,s)$, for some integer $s$. Moreover, $v^B(\ke_j)=(2,\ell+jf^*h+2B,s_j)$, for some $\ell\in\Pic(S)$, $s_j\in\ZZ$ and $j=0,1$.
\end{lem}

\begin{proof}
Take a line $l\subseteq Y$ such that $l\not\subseteq P$.
Then $\kj_l$ is a pure torsion-free sheaf supported on a curve in the linear system of $f^*h$. Assuming $l$ generic with the above property, we necessarily have that $v^B(\kj_l)$ is as required.
As the Mukai vector does not depend on the choice of the line, this holds for any line $l\subseteq Y$.

To prove the second assertion in the statement, we use this and the fact that, by Proposition \ref{prop:ideallines}, for $l\subseteq P\subseteq Y$, the complex $\kj_l$ is extension of $\e[1]$ and $\f$. Thus it is enough to prove that $v^B(\e)=(2,\ell+2B,s_0)\in\Pic(S,B)$, for $\ell\in\Pic(S)$. This follows directly from Lemma \ref{lem:NS}.
\end{proof}

For simplicity, we will use the same symbol $\Xi\colon\Db(S,\beta)\to\Db(\PP^2,\kb_0)$ for the derived functor of the one in \eqref{eqn:equiv}. We denote by $\Xi_L$ the left adjoint of $\Xi$.
We will also denote by $\Theta$ the composition of functors
\[
\Theta:=\Xi_L\circ(\sigma_*\circ\Phi'')^{-1}\circ R_{\ko_{Y}(-H)}\circ((-)\otimes\ko_Y(-H)),
\]
where $(\sigma_*\circ\Phi'')^{-1}$ is the functor defined in Section \ref{subsec:semiorthogonal}, and $R_{\ko_{Y}(-H)}$ is the right mutation in $\ko_Y(-H)$.
The functor $\Theta:\cat{T}_Y\to\Db(S,\beta)$ is an equivalence.

For the convenience of the reader, we put all the functors defined so far and providing different incarnations of the triangulated category $\cat{T}_Y$ in the following diagram
\[
\xymatrix{
\cat{T}_Y\ar[dd]_-{\Theta}\ar[rrrr]^-{R_{\ko_{Y}(-H)}\circ((-)\otimes\ko_Y(-H))}&&&&\cat{T}_Y\ar[rr]^-{(\sigma_*\comp\Phi'')^{-1}}&&\Db(\PP^2,\kb_0)\ar@<-1ex>[ddllllll]_-{\Xi_L}\\
&&&\\
\Db(S,\beta)\ar[uurrrrrr]_-{\Xi}\ar[rrrrrr]^-{(-)\otimes_{\ko_S}\e\dual}&&&&&&\Db(S,\ka_0),\ar[uu]_-{f_*}
}
\]
where the upper and lower triangles commute.

\subsection{Birationality I}\label{subsec:bir1}

Following for example \cite{Y}, a notion of stability for twisted sheaves can be introduced and the moduli spaces of twisted stable sheaves can be constructed. For a K3 surface $S$, given a primitive $v\in\Pic(S,B)$ with $\langle v,v\rangle\geq2$, we denote by
\[
M(S,v,B)
\]
the moduli space of $\beta$-twisted stable (with respect to a generic polarization) sheaves on $S$ with twisted Mukai vector $v$ (here we specify the lift $B$ of $\beta$). By \cite{Y}, $M(S,v,B)$ is a smooth projective holomorphic symplectic manifold and, as lattices,
\[
\co^2(M(S,v,B),\ZZ)\cong v^\perp,
\]
where the orthogonal is taken in $\widetilde \co(S,B,\ZZ)$.
The weight-$2$ Hodge structure on $v^\perp$, which is compatible with this isometry, is the one induced by the twisted period $\sigma_S+\sigma_S\wedge B$, for $\CC\cdot\sigma_S=\co^{2,0}(S)$. In other words, the $\co^{2,0}$-part of the weight-$2$ Hodge structure on $v^\perp$ is given by $\CC(\sigma_S+\sigma_S\wedge B)$, as $\langle v,\sigma_S+\sigma_S\wedge B\rangle=0$.

The following proposition proves the second statement in Theorem \ref{thm:main2}. We keep the same notation as before and denote by $P\dual\subseteq F(Y)$ the dual plane of lines contained in the plane $P$.

\begin{prop}\label{prop:bir1}
	Let $Y$ be a cubic fourfold satisfying $(\ast)$. Then there exists a birational map $F(Y)\dasharrow M(S,v,B)$, where $v=(0,f^*h,s)\in\Pic(S,B)$ for some $s\in\ZZ$, which is either an isomorphism or a Mukai flop in the plane $P\dual$.
\end{prop}

\begin{proof}
By Remark \ref{rmk:modsp}, $F(Y)$ is identified with the moduli space of the stable sheaves $\kf_l$, for $l\subseteq Y$ a line.
By definition,  $\Theta(\kf_l)=\kj_l$. On the other hand, by Proposition \ref{prop:ideallines}, if $l\not\subseteq P$, then the pure torsion sheaf $\kj_l\in\coh(S,\beta)$ is stable, since its Mukai vector is primitive (see Lemma \ref{lem:calc}) and, by $(\ast)$, the divisor $f^*h$ is indecomposable.

Let $P\dual\subseteq F(Y)$ be the plane dual to $P\subseteq Y$ and parametrizing the lines in $P$. The above remarks show that $\Theta$ yields an isomorphism between the open subset $U=F(Y)\setminus P\dual$ and an open subset $U'\subseteq M(S,v,B)$, where $v=(0,f^*h,s)$. Therefore, we have a birational map $\gamma:F(Y)\dashrightarrow M(S,v,B)$.

Fix $h'\in\Pic(M(S,v,B))$ an ample polarization. Suppose that $\gamma$ does not extend to an isomorphism. Hence, by \cite[Prop.\ 2.1 and Cor.\ 2.2]{Huy1}, for all rational curves $D\subseteq P\dual$, we have $\gamma^*h'\cdot D<0$ (see also \cite[Cor.\ 2.6]{Huy1}). Consider the Mukai flop $\gamma':M\dashrightarrow F(Y)$ of $F(Y)$ in the plane $P\dual$ and let $P'\subseteq M$ be the corresponding dual plane in $M$. Let $\gamma'':M\dashrightarrow M(S,v,B)$ be the composition $\gamma\circ\gamma'$. Since a Mukai flop reverses the sign of the intersection with the rational curves contained in the exceptional locus, the above remarks give that $(\gamma'')^*h'\cdot D'>0$, for any rational curve $D'$ in $P'$. Again by \cite[Prop.\ 2.1 and Cor.\ 2.2]{Huy1}, this means that $\gamma''$ extends to an isomorphism yielding the desired conclusion.
\end{proof}

The fact that we have to deal with moduli spaces of twisted sheaves with (in general) non-trivial twists is quite relevant. This is made clear by Proposition \ref{prop:nospmod}. Moreover, we will see in Remark \ref{rmk:MukaiFlopUtah} that, if $\beta$ is non-trivial, then the birational map in Proposition \ref{prop:bir1} is actually a Mukai flop.

The following easy corollary is a specialization of the previous result to a dense subset of the moduli space $\kc$ of cubic fourfolds (see \cite[Sect. 2.2]{H2} for more details about the GIT construction of $\kc$). A cubic fourfold $Y$ containing a plane $P$ is \emph{very generic} if
\[
\NS_2(Y):=\co^4(Y,\ZZ)\cap \co^{2,2}(Y)=\langle H^2,P\rangle,
\]
where, for simplicity, we denote by $H^2$ the self-intersection of the ample line bundle $\ko_Y(H)$. It is easy to see that a very generic $Y$ satisfies $(\ast)$ (see, for example, \cite{BMMS} and Remark \ref{rmk:stupid3}).

\begin{cor}\label{cor:birgen}
	If $Y$ is very generic, then there exists a birational morphism $F(Y)\dasharrow M(S,v,B)$, where $v=(0,f^*h,0)$, which is either an isomorphism or a Mukai flop in the plane $P\dual$.
\end{cor}

\begin{proof}
By Proposition \ref{prop:bir1} and by tensoring with an appropriate power of $f^*h\in\Pic(S)$, one deduces that there exists a birational morphism $F(Y)\dasharrow M(S,v,B)$, where $v=(0,f^*h,\epsilon)$ and $\epsilon\in\{0,1\}$. Indeed, the tensor product with $f^*h$ changes $v$ in $v+(0,0,2)$.
By \cite{BD}, $\Pic(F(Y))$ contains a line bundle $L$ with self-intersection $6$ with respect to the Beauville--Bogomolov quadratic form. Moreover, by \cite{Y} and the existence of the above birational map, we have a sequence of isometries of lattices
\begin{equation}\label{eqn:NS1}
\Pic(F(Y))\cong\Pic(M(S,v,B))\cong\Pic(S,B)\cap v^\perp.
\end{equation}

Therefore, it is enough to show that, if $\epsilon=1$, then $\Pic(M(S,v,B))$ does not contain an element of self-intersection $6$.
For this we use the calculations in \cite{vG1}.
Indeed, denote by $e_1,e_2$ the two natural generators of the first copy of the hyperbolic lattice in the K3 lattice $\Lambda:=U^{\oplus 3}\oplus E_8(-1)^{\oplus 2}$ (see \cite{Ni} for the definition of the lattices $U$ and $E_8$). We can identify $\co^2(S,\ZZ)$ to $\Lambda$ in such a way that $f^*h$ is represented by the vector $e_1+e_2$. By \cite{vG} and using this identification, the lift $B$ of $\beta$ has the form $B=\frac{e_2+\lambda}{2}$, for some $\lambda\in U^{\oplus 2}\oplus E_8(-1)^{\oplus 2}$ and $(\lambda,\lambda)=2$. If $\epsilon=1$, then, by the last isometry in \eqref{eqn:NS1}, an easy calculation (using Lemma \ref{lem:NS}) shows that $\Pic(M(S,v,B))\cong\langle (0,0,1),(4,2\lambda+e_1+2e_2,1)\rangle$. Hence the intersection form on $\Pic(M(S,v,B))$ would not represent $6$ when $\epsilon=1$.
\end{proof}

\subsection{Birationality II: end of the proof of Theorem \ref{thm:main2}}\label{subsec:bir2}

Let us now consider the general case of $Y$ any cubic fourfold containing a plane $P$. In particular, we prove that, for all cubic fourfolds $Y$ containing a plane, the Fano variety $F(Y)$ is birational to a smooth projective moduli space of twisted sheaves on a K3 surface which, together with the generic case in the previous section, is precisely the content of Theorem \ref{thm:main2}. The argument is divided into a few steps.

\subsubsection{Cohomologies of cubic fourfolds and Fano varieties}\label{subsubsec:lattice}

Let us start by recalling some lattice theoretic properties of the cohomologies of the cubic fourfolds containing a plane, their Fano varieties of lines and the associated K3 surfaces. A detailed explanation of what we are about to discuss is in \cite{H2} (see also \cite{V}).

If $Y$ is a cubic fourfold, by \cite[Prop.\ 6]{BD}, there exists a Hodge isometry
\begin{equation}\label{eqn:fano1}
	\varphi:\co^2(F(Y),\ZZ)_\mathrm{prim}\isomor\co^4(Y,\ZZ)_\mathrm{prim}(-1),
\end{equation}
where $(-1)$ just reverses the signature.
Here $\co^4(Y,\ZZ)_\mathrm{prim}$ is the orthogonal complement (with respect to the cup-product) of the self-intersection $H^2$ of $\ko_Y(H)$. Analogously, $\co^2(F(Y),\ZZ)_\mathrm{prim}$ is the orthogonal complement (with respect to the Beauville--Bogomolov form) of the natural ample polarization coming from the embedding $F(Y)\hookrightarrow\mathrm{Gr}(2,6)$ into the Grassmannian of lines in $\PP^5$.

Let $L$ be the lattice $(1)^{\oplus 21}\oplus(-1)^{\oplus 2}$ with the choices of a distinguished class $\h^2$ of self-intersection $3$ and of an isometry $\psi:\co^4(Y,\ZZ)\to L$ sending $H^2$ to $\h^2$. Notice that all the vectors in $L$ of self-intersection $3$ are permuted by the isometries in the orthogonal group $\mathrm{O}(L)$. Set $L_0:=(\h^2)^\perp=\psi(\co^4(Y,\ZZ)_\mathrm{prim})$, $K:=\psi(\langle H^2,P\rangle)$ and $L_1:=K^\perp\subseteq L_0$, where the orthogonals are taken in $L$. As explained in \cite[Prop.\ 3.2.4]{H2}, we have the following fact.

\begin{lem}\label{lem:embprim}
The primitive embedding of $L_1$ into $L$ is unique, up to isometries of $L$ acting as the identity on $\h^2$.
\end{lem}

Let $Y'$ be a very generic cubic fourfold with a plane with an isometry $\psi:\co^4(Y',\ZZ)\to L$ as above. Then
\begin{equation}\label{eqn:latt1}
L_1=\psi(T(Y'))=\psi(\varphi(T(F(Y')))(-1)),
\end{equation}
where $T(Y')\hookrightarrow \co^4(Y',\ZZ)$ and $T(F(Y'))\hookrightarrow \co^2(F(Y'),\ZZ)$ are the smallest primitive sublattices with the property that $\sigma_Y\in T(Y')\otimes\CC$ and $\sigma_{F(Y')}\in T(F(Y'))\otimes\CC$. Here $\CC\sigma_{Y'}=\co^{3,1}(Y')$ and $\CC\sigma_{F(Y')}=\co^{2,0}(F(Y'))$. In other words, $T(Y')=(\sigma_{Y'}^\perp\cap \co^4(Y,\ZZ))^\perp$ and similarly for $T(F(Y'))$ (here the orthogonality is in $\co^4(Y',\ZZ)$ and $\co^2(F(Y'),\ZZ)$ respectively).

\subsubsection{Associated twisted K3 surfaces}\label{subsubsec:K3s}

Take $Y'$ to be again a fixed very generic cubic fourfold with a plane $P$. At the same time, we fix an isometry $\psi\colon\co^4(Y',\ZZ)\to L$ such that $\psi(H^2)=\h^2$ which, by \eqref{eqn:fano1}, gives isometries $\psi':\co^2(F(Y'),\ZZ)_\mathrm{prim}\to L_0(-1)$ and $\psi'':T(F(Y'))\to L_1(-1)$.

By \cite[Sect.\ 1]{V} (use, in particular, \cite[Prop.\ 2]{V} and the fact that $T(Y')$ is orthogonal to $H^2$ and $P$ in $\co^4(Y',\ZZ)$), there exists an index-$2$ embedding
\begin{equation}\label{eqn:FanoK3}
\xi:T(F(Y'))\cong T(Y')(-1)\hookrightarrow T(S)\cong\co^2(S,\ZZ)_\mathrm{prim},
\end{equation}
where $S$ is the double cover of $\PP^2$ ramified along the smooth sextic associated to $Y'$ and $P$. Here $\co^2(S,\ZZ)_\mathrm{prim}$ is the orthogonal complement of $f^*h$ in $\co^2(S,\ZZ)$. Here and for the rest of the section, we keep the same notation as in Section \ref{subsec:geometry}.

By Corollary \ref{cor:birgen}, there exists a birational map
\begin{equation}\label{eqn:morbir}
\xymatrix{\gamma:F(Y')\ar@{-->}[r]& M(S,v,B)},
\end{equation}
where $M(S,v,B)$ is the moduli space of $\beta$-twisted sheaves with Mukai vector $v:=(0,f^*h,0)$. Hence, there is a Hodge isometry $\gamma^*:\co^2(M(S,v,B),\ZZ)\isomor \co^2(F(Y'),\ZZ)$. In particular, $$T(M(S,v,B)))=(\gamma^*)^{-1}(T(F(Y'))),$$ where $T(M(S,v,B))$ is the minimal primitive sublattice of $\co^2(M(S,v,B),\ZZ)$ whose tensorization by $\CC$ contains $\co^{2,0}(M(S,v,b))$. The lattice $\co^2(M(S,v,B),\ZZ)$ is naturally identified to $v^\perp$, with the orthogonal taken in $\widetilde{\co}(S,\ZZ)$ with respect to the Mukai pairing (see \cite[Thm.\ 3.19]{Y}). More precisely, there is a natural Hodge isometry $$\kappa:\co^2(M(S,v,B),\ZZ)\to v^\perp,$$ where the Hodge structure on $v^\perp$ is the one induced by $T(S,B)\subseteq v^\perp$. In particular, we have $\kappa(T(M(S,v,B)))=T(S,B)$.

Denote by $\Lambda_0\hookrightarrow\Lambda:=U^{\oplus 3}\oplus E_8(-1)^{\oplus 2}$ the orthogonal complement of a (primitive) vector $k$ in the K3 lattice $\Lambda$ with $(k,k)=2$. Notice that we can choose any such $k$ as they are all conjugate under the action of the group $\mathrm{O}(\Lambda)$ of isometries of $\Lambda$ (this is essentially because, due to \cite{Ni}, the primitive embedding of a rank-$1$ lattice in $\Lambda$ is unique, up to the action of $\mathrm{O}(\Lambda)$). Hence, there is an isometry $\eta:\co^2(S,\ZZ)\to\Lambda$ such that $\eta((f^*h)^\perp)=\Lambda_0$.
If we define $\widetilde\Lambda:=\Lambda\oplus U$, then there exists an isometry $\widetilde\eta:\widetilde\co(S,\ZZ)\to\widetilde\Lambda$ such that $\widetilde\eta|_{\co^2(S,\ZZ)}=\eta$ and $w:=\widetilde\eta(v)$ can be written as $w=(0,k,0)$ as a vector in $\widetilde\Lambda$. Here the trivial coordinates refer to the vectors $e_1$ and $e_2$ in the standard basis of the additional copy of $U$.

Let us now define two primitive embeddings $i_1:L_1(-1)\hookrightarrow w^\perp$ and $i_2: L_0(-1)\hookrightarrow w^\perp$ as follows. Take the isometry $\phi:=\widetilde\eta\circ\kappa\circ(\gamma^*)^{-1}:\co^2(F(Y'),\ZZ)\to w^\perp$ and put $$i_1:=\phi\circ(\psi'')^{-1}:L_1(-1)\hookrightarrow w^\perp\qquad\mbox{and}\qquad i_2:=\phi\circ(\psi')^{-1}:L_0(-1)\hookrightarrow w^\perp.$$ It is clear that $i_1(L_1(-1))\subseteq i_2(L_0(-1))\subseteq w^\perp$.

\begin{lem}\label{lem:FanoK3}
	Let $Y$ be a cubic fourfold with a plane. Then there is an isometry $\phi:\co^2(F(Y),\ZZ)\to w^\perp$ such that $\phi(T(F(Y)))\subseteq i_1(L_1(-1))\subseteq w^\perp$.
\end{lem}

\begin{proof}
	By the above discussion there is already a very generic cubic fourfold $Y'$ for which such a $\phi$ exists. Using this case and since all Fano varieties $F(Y)$ of cubic fourfolds $Y$ containing a plane belong to the same irreducible component of the moduli space of (primitively polarized) irreducible holomorphic symplectic manifolds, one gets an isometry $\phi':\co^2(F(Y),\ZZ)\to w^\perp$ such that $\phi'(\co^2(F(Y),\ZZ)_\mathrm{prim})=i_2(L_0(-1))\subseteq w^\perp$, for all cubic fourfolds $Y$ with a plane. This follows, for example, from \cite[Thm.\ 1.5]{HGS} (see also \cite{vG1}).
	
	By Lemma \ref{lem:embprim}, there is a isometry $\xi_1\in\mathrm{O}(L_0(-1))$ such that $\xi_1(\phi'(T(F(Y))))\subseteq i_1(L_1(-1))$. Moreover $\xi_1$ acts trivially on the discriminant group $A_{L_0(-1)}:=L_0(-1)\dual/L_0(-1)$ and so, by Sections 1.14 and 1.15  of \cite{Ni}, there is $\xi_2\in\mathrm{O}(w^\perp)$ such that $\xi_2|_{i_2(L_0(-1))}=\xi_1$. Thus define $\phi=\xi_2\circ\phi'$.
\end{proof}

\subsubsection{The end of the proof of Theorem \ref{thm:main2}}\label{subsubsec:end}

Let us begin with the following result.

\begin{lem}\label{lem:FanoK3v2}
	Let $Y$ be a cubic fourfold with a plane. Then there exist a K3 surface $S'$, a Mukai vector $v'\in\widetilde\co(S',\ZZ)$ and $B'\in\co^2(S',\QQ)$ such that there is a Hodge isometry $\co^2(F(Y),\ZZ)\cong\co^2(M(S',v',B'),\ZZ)$.
\end{lem}

\begin{proof}
	Using Lemma \ref{lem:FanoK3}, fix an isometry $\phi:\co^2(F(Y),\ZZ)\to w^\perp$ such that $\phi(T(F(Y)))\subseteq i_1(L_1(-1))$. Put $\sigma_1:=\phi_\CC(\sigma_{F(Y)})\in i_1(L_1(-1))\otimes\CC$, where $\phi_\CC$ is the $\CC$-linear extension of $\phi$, and set $\sigma_2$ to be the component of $\sigma_1$ contained in $\Lambda\otimes\CC$. Notice that, as $\sigma_1\in w^\perp\otimes\CC$, the vector $\sigma_2$ is orthogonal to $k\in\Lambda$ and therefore $\sigma_2\in\Lambda_0\otimes\CC\subseteq\Lambda\otimes\CC$.
	
	By construction $\sigma_2$ induces a weight-$2$ Hodge structure on $\Lambda_0$ and, by the surjectivity of the period map for K3 surfaces, there is a K3 surface $S'$ and an isometry $\widetilde\eta':\widetilde\co(S',\ZZ)\to\widetilde\Lambda$ as explained in Section \ref{subsubsec:K3s} (indeed, to define $\widetilde\eta$ for the K3 surface $S$, we did not use that $\Pic(S)\cong\ZZ$). In particular, $\widetilde\eta'(T(S'))$ is a primitive sublattice of $\Lambda_0$.
	
We then have a commutative diagram
	\begin{equation}\label{eqn:diagr}
	\xymatrix{
	T(S')\ar@{^{(}->}[d]_{\widetilde\eta'|_{T(S')}}\ar[r]^{\langle-,B'\rangle}&\ZZ/2\ZZ\ar@{=}[d]&\\
	\Lambda_0\ar[r]^-{(-,\widetilde{B})}&\ZZ/2\ZZ\ar[r]&0,
	}
	\end{equation}
where $(-,-)$ denotes the bilinear pairing on $\Lambda$ and $\langle-,-\rangle$ is the Mukai pairing. Notice that the bottom line is obtained by applying $\widetilde\eta$ to the surjection
	\begin{equation}\label{eqn:l1l1}
\xymatrix{
T(S)\ar[r]^-{\langle-,B\rangle}&\ZZ/2\ZZ\ar[r]&0,
}
\end{equation}
whose existence is due to \cite{C}, being $B$ a B-field lift of $\beta\in\Br(S)$. In particular, in \eqref{eqn:diagr}, we put $\widetilde{B}:=(\widetilde\eta_\QQ)(B)$ and $B':=(\widetilde\eta'_\QQ)^{-1}(\widetilde{B})$.

Set $v':=(\widetilde\eta')^{-1}(w)$ and consider the moduli space $M(S',v',B')$ with the natural identification $\kappa':\co^2(M(S',v',B'))\isomor (v')^\perp$ as discussed in Section \ref{subsubsec:K3s}.
By \cite[Thm.\ 3.19]{Y}, there is $c\in\CC$ such that $(\widetilde\eta'\circ\kappa')_\CC(c\sigma_{M(S',v',B')})=\widetilde\eta'_\CC(\sigma_{S'}+\sigma_{S'}\wedge B')=\sigma_2+\sigma_2\wedge\widetilde{B}=\sigma_1$.

Thus the isometry $(\kappa')^{-1}\circ(\widetilde\eta')^{-1}\circ\phi:\co^2(F(Y),\ZZ)\isomor\co^2(M(S',v',B'),\ZZ)$ preserves the Hodge structures. This is precisely what we want.
\end{proof}

To conclude the proof of Theorem \ref{thm:main2}, observe that the Fano varieties $F(Y)$ of cubic fourfolds $Y$ containing a plane and the moduli spaces of twisted sheaves on a K3 surface are both deformations of Hilbert schemes of $0$-dimensional subschemes of length-$2$ on a K3 surface (see \cite{BD} and \cite{Y}). Thus we can apply the birational Torelli theorem for those manifolds. More precisely, given a cubic fourfold $Y$ containing a plane $P$ and the Hodge isometry $\co^2(F(Y),\ZZ)\cong\co^2(M(S',v',B'),\ZZ)$ as in Lemma \ref{lem:FanoK3v2}, by the main result in \cite{Ve} (see also \cite{HuyV}), we get a birational map $F(Y)\dashrightarrow M(S',v',B')$ as required. Notice that here the monodromy calculations in \cite{M1,M2} are crucial.

\begin{remark}\label{rmk:referee}
As suggested by one of the referees, one may give a different geometric argument to construct the family of K3 surfaces associated to the family of cubic fourfolds with a plane. Indeed, take a family $\ky\to\kd$ of cubic fourfolds containing a plane. Let $\pi:\widetilde\ky\to\PP^2_\kd$ be the map induced by the projection from $P_\kd$ onto $\PP_\kd^2$, blown-up in $P_\kd$. Note that the Stein factorization of the relative Hilbert scheme of lines of $\widetilde\ky\to\PP^2_\kd$ gives a family $\ks_0\to\kd$ of K3 surfaces over $\kd$ with singular fibres over the points of $\kd$ corresponding to nodal sextics.

Passing to a double cover of $\kd$, one gets a family $\ks'$ with a locus of ordinary double points along the set of singular points in the singular fibres (i.e., in codimension $3$). Therefore, choosing a small analytic resolution of singularities $\ks''\to\ks'$ we get a family $\ks''\to\kd'$ of smooth K3 surfaces. The fibres of this family are double covers of $\PP^2$ blown-up in the nodes of the sextic.
\end{remark}

\section{Rationality and moduli spaces}\label{subsec:bir3}

In this section we relate the fact that $F(Y)$ is birational to a moduli spaces of untwisted sheaves to the existence of exact equivalences between the category $\cat{T}_Y$ and the bounded derived category of untwisted sheaves on some K3 surface. According to Kuznetsov's conjecture mentioned in the introduction, this would provide a comparison between two points of view on rationality of a cubic fourfold with a plane: one given by looking at the birational geometry $F(Y)$ and the other at the category $\cat{T}_Y$.

\smallskip

We start with the following result that is certainly well-known to the experts.

\begin{prop}\label{prop:nospmod}
	If $Y$ is a very generic cubic fourfold containing a plane $P$, then there is no smooth projective K3 surface $S'$ and no primitive Mukai vector $v'$ such that $F(Y)$ is birational to the smooth projective moduli space of stable untwisted sheaves on $S'$ with Mukai vector $v'$.
\end{prop}

\begin{proof}
Assume, by contradiction, that there exists a K3 surface $S'$ and a primitve Mukai vector $v\in\widetilde \co(S',\ZZ):=\widetilde \co(S',0,\ZZ)$ such that $F(Y)$ is birational to the moduli space $M(S',v)$ of stable untwisted sheaves with Mukai vector $v$. Then there would be a Hodge isometry $T(F(Y))\cong T(M(S',\ZZ))\cong T(S')$ (see \cite{Y}). Hence $T(F(Y))$ would admit a primitive embedding in the K3 lattice $\Lambda$.

Now, due to \eqref{eqn:FanoK3}, \eqref{eqn:l1l1} and \cite{V}, we have a short exact sequence
\begin{equation}\label{eqn:fano2}
\xymatrix{
0\ar[r]&T(F(Y))\ar[r]&T(S)\ar[r]^{\langle-,B\rangle}&\ZZ/2\ZZ\ar[r]&0,
}
\end{equation}
where $S$ is again the K3 surface double cover of $\PP^2$ and $B$ is the lift of the special non-trivial $\beta\in\Br(S)$ in Section \ref{subsec:geometry}. But, due to either \cite[Cor.\ 9.4 and Sect.\ 9.7]{vG} or to the appendix of \cite{Kuz3}, the kernel of the morphism $(-,B)$ in \eqref{eqn:fano2} does not admit a primitive embedding in $\Lambda$.
\end{proof}

Nevertheless $F(Y)$ may be birational to a moduli space of untwisted sheaves.
As explained in the following proposition, which answers a question by Hassett in the case of cubic fourfolds containing a plane, this is reflected by the category $\cat{T}_Y$.

\begin{thm}\label{prop:Hassett}
Let $Y$ be a cubic fourfold satisfying $(\ast)$. There exists a K3 surface $S'$ such that  $F(Y)$ is birational to a fine moduli space $M(S',v)$ of untwisted stable sheaves on $S'$ and with Mukai vector $v$ if and only if $\cat{T}_Y\cong\Db(S')$.
\end{thm}

\begin{proof}
By assumption $\cat{T}_Y\cong\Db(S,\beta)$, where $S$ and $\beta\in\Br(S)$ are as the previous sections.
Under the assumption that $F(Y)$ and $M(S',v)$ are birational, there exists a Hodge isometry $T(F(Y))=\Pic(F(Y))^\perp\cong T(M(S',v)):=\Pic(M(S',v))^\perp$, where the orthogonal is taken with respect to the Beauville--Bogomolov form on the second cohomology group. Moreover, there are Hodge isometries $T(M(S',v))\cong T(S')$ (see, for example, \cite{Y} or the beginning of Section \ref{subsec:bir1}) and, as a consequence of Proposition \ref{prop:bir1}, $T(F(Y))\cong T(S,\beta)$. Putting the isometries together, we get a Hodge isometry $T(S')\cong T(S,\beta)$.

By \cite[Thm.\ 1.14.4]{Ni}, this isometry extends to a Hodge isometry $\varphi:\widetilde \co(S',\ZZ)\cong\widetilde \co(S,B,\ZZ)$, where $B$ is a $B$-field lift of $\beta$. Hence, up to composing $\varphi$ with the isometry $-\id_{\co^0(S',\ZZ)\oplus \co^4(S',\ZZ)}\oplus \id_{\co^2(S',\ZZ)}$, we can apply the main result in \cite{HS2} obtaining an equivalence $\Db(S')\cong\Db(S,\beta)$. In conclusion, $\Db(S')\cong\cat{T}_Y$.

Conversely, assume that there exist a K3 surface $S'$ and an exact equivalence $\Db(S')\cong\cat{T}_Y$. Due to $(\ast)$, there is a an equivalence $\Phi:\Db(S,\beta)\isomor\Db(S')$ which, in view of \cite{CS}, is a Fourier--Mukai functor $\Phi_\ke$. By \cite{HS1}, this equivalence induces a Hodge isometry $\Phi_\ke^H:\widetilde \co(S,B,\ZZ)\to\widetilde \co(S',\ZZ)$ for a $B$-field lift $B$ of $\beta$.

Due to Proposition \ref{prop:bir1}, $F(Y)$ is birational to a moduli space $M(S,v,B)$, for a primitive vector $v\in\widetilde \co(S,B,\ZZ)$. Set $v':=\Phi^H_\ke(v)$. Up to composing $\Phi_\ke$ with the shift by $1$, we can assume that the degree zero part of $v'$ is greater or equal to $0$ and so we can consider the (non-empty) moduli space $M(S',v')$ of stable untwisted sheaves on $S'$. By \cite{Y}, there are Hodge isometries
\[
\co^2(M(S,v,B),\ZZ)\cong v^\perp\cong (v')^\perp\cong \co^2(M(S',v'),\ZZ).
\]
Notice that both $M(S,v,B)$ and $M(S',v')$ are deformation equivalent to a Hilbert scheme of length-$2$ subschemes of dimension $0$ on a K3 surface (\cite{Y}). Hence, applying the Torelli theorem for hyperk\"{a}hler manifolds (see \cite{Ve} and \cite{HuyV}), we get a birational map $M(S,v,B)\dashrightarrow M(S',v')$. Putting all together, we get a birational map $F(Y)\dashrightarrow M(S',v')$, as required.
\end{proof}

\begin{remark}\label{rmk:Hassett2} (i) A special but interesting case of the previous proposition occurs if $Y$ satisfies $(\ast)$ and the moduli space is actually isomorphic to $\mathrm{Hilb}^2(S')$, the Hilbert scheme of length-$2$ $0$-dimensional subschemes of a K3 surface $S'$. The main conjecture in \cite{Kuz3} asserts that a cubic fourfold with a plane $Y$ is rational if and only if there exists a K3 surface $S'$ and an equivalence $\cat{T}_Y\cong\Db(S')$. Thus, conjecturally, if $F(Y)$ is birational to a Hilbert scheme as above, then $Y$ is rational.
	
It would be certainly interesting to remove the assumption $(\ast)$ but this would require a different approach exceeding the scope of this paper.
	
(ii) Assuming $\rk(\NS_2(Y))>12$ and using \cite{Ve}, the statement of Theorem \ref{prop:Hassett} can be extended further. Indeed, due to our assumption on the rank of $\NS_2(Y)$, by \cite[Prop.\ 7.3]{HS1} there exist a K3 surface $S'$ and an equivalence $\cat{T}_Y\cong\Db(S')$. Again, conjecturally, this would mean that all cubic fourfolds with a plane and such that $\rk(\NS_2(Y))>12$ are rational.

From \cite[Thm.\ 0.4]{HS1} we deduce the sequence of Hodge isometries $T(F(Y))\cong T(S,B)\cong T(S')$. By \cite{Be} there is a Hodge isometry $T(S')\cong T(\mathrm{Hilb}^2(S'))$ and such an isometry extends to a Hodge isometry $\co^2(F(Y),\ZZ)\cong \co^2(\mathrm{Hilb}^2(S'),\ZZ)$. For this use \cite{Ni} with our assumption on the rank of $T(Y)$ (and so of $T(S')$) and the fact that $\co^2(F(Y),\ZZ)$ and $\co^2(\mathrm{Hilb}^2(S'),\ZZ)$ contain a copy of the K3 lattice $\Lambda$ into which $T(F(Y))$ and $T(S')$ embed (for this use \cite{Be}).

Using the main result of \cite{Ve} (see also \cite{HuyV}), the Hodge isometry $\co^2(F(Y),\ZZ)\cong \co^2(\mathrm{Hilb}^2(S'),\ZZ)$ would yield a birational map between $F(Y)$ and $\mathrm{Hilb}^2(S')$.
\end{remark}

\section{Isomorphisms and moduli spaces of complexes}\label{sec:iso}

In the course of this section, unless clearly specified, we will always assume that all cubic fourfolds $Y$ have the restrictive property
\begin{itemize}
\item[($\ast\ast$)] $Y$ satisfies $(\ast)$ and $\beta\in\Br(S)$ is non-trivial.
\end{itemize}
As observed in the appendix to \cite{Kuz3}, the very generic cubic fourfolds containing a plane satisfy the additional condition about $\beta$. Moreover, as it follows from \cite[Prop.\ 4.7]{Kuz3}, the cubic fourfolds for which $(\ast\ast)$ does not necessarily hold true are precisely those in the codimension-$1$ subvarieties of the divisor $\kc_8$ consisting of rational cubic fourfolds with a plane studied in \cite{Ha1}.

Although relative to a different setting, most of the arguments here are inspired by \cite{AB}.

\subsection{A family of stability conditions}\label{subsec:definigstab}

Given the twisted Chern character $\ch^B$, for any torsion-free $\kf\in\coh(S,\beta)$ define its \emph{slope} in the following way:
\[
\mu^B(\kf):=\frac{c_1^B(\kf)\cdot f^*h}{\rk(\kf)}.
\]
In particular a torsion-free $\kf\in\coh(S,\beta)$ is \emph{$\mu^B$-semistable} (resp.\ \emph{$\mu^B$-stable}) if, for all non-zero $\kg\hookrightarrow\kf$ in $\coh(S,\beta)$, $\mu^B(\kg)\leq\mu^B(\kf)$ (resp.\ $<$ and $\rk(\kg)<\rk(\kf)$). Notice that, in the definition of slope, the ample polarization has been fixed to be equal to $f^*h$.

The following result could be proved for cubic fourfolds satisfying $(\ast)$ instead of $(\ast\ast)$. Since we will not need this generality, we just consider $(\ast\ast)$ for which the proof is straightforward.

\begin{lem}\label{prop:stab}
	The sheaves $\e$ and $\f$ are $\mu^B$-stable. Moreover, any morphism $\e\to\f$ is injective.
\end{lem}

\begin{proof}
Since, by assumption, $\beta$ is non-trivial, the two claims follow directly from Corollary \ref{cor:rankdiv}.
\end{proof}

\smallskip

As we observed in \cite{HMS}, the notion of twisted Chern character allows one to generalize Bridgeland's construction in \cite{Br1} to the case of twisted K3 surfaces.

We skip all the details about the definition of stability conditions which are not relevant in this paper and for which we refer to \cite{Br,HMS}. So we just recall that, by \cite[Prop.\ 5.3]{Br} giving a stability condition on a triangulated category $\cat{D}$ is equivalent to giving a bounded $t$-structure on $\cat{D}$ with heart $\cat{A}$ and a group homomorphism $Z:K(\cat{A})\to\CC$ such that $Z(\kg)\in\HH$, for all $0\neq \kg\in\cat{A}$, and with Harder--Narasimhan filtrations (see \cite[Sect.\ 5.2]{Br}). Here $\HH:=\{ z\in\CC^*:z=|z|\exp(i\pi\phi), \, 0< \phi \leq 1 \}$. More precisely, any $0\neq\kg\in\cat{A}$ has a well-defined \emph{phase} $\phi(\kg):=\mathrm{arg}(Z(\kg))\in(0,1]$.
For $\phi\in(0,1]$, we denote by $\kp(\phi)$ the category of $\phi$-semistable objects in $\cat{A}$ of phase $\phi$: more precisely, an object $0\neq\kg\in\cat{A}$ is then in $\kp(\phi)$ if and only if, for all $\kg\epi\kg'\neq 0$ in $\cat{A}$, $\phi=\phi(\kg)\leq\phi(\kg')$. If $\phi\in\RR$, then there is a unique $n\in\ZZ$ and $\phi'\in(0,1]$ such that $\phi=\phi'+n$. So set $\kp(\phi):=\kp(\phi')[n]$.

A stability condition is called \emph{locally-finite} if there
exists some $\epsilon > 0$ such that, for all $\phi\in\RR$, each
(quasi-abelian) subcategory  $\mc{P} ((\phi - \epsilon , \phi +
\epsilon))$ (i.e., the category generated by extensions by all semistable objects with phases in the interval $(\phi-\epsilon,\phi+\epsilon)$) is of finite length.
In this case $\mc{P} (\phi)$ has finite length so that every object in $\mc{P} (\phi)$ has a finite
\emph{Jordan--H\"older filtration} into stable factors of the same phase.

\smallskip

Assume from now on that $Y$ is a cubic fourfold containing a plane $P$ and satisfying $(\ast\ast)$. We want to produce a family of stability conditions on the derived category $\Db(S,\beta)$, where $(S,\beta)$ is the twisted K3 surface associated to $Y$. Observe that, by the definition of a lift of a Brauer class, we can assume further that the $B$-field lift of $\beta$ is such that $B\cdot f^*h=\frac{1}{2}$ (see \cite[Lemma 6.4]{Kuz3}). By Lemma \ref{lem:calc}, we can write $v^B(\e)=(2,\ell+2B,s_0)\in\widetilde \co(S,B,\ZZ)$. Define
\[
Z_{m}:\Pic(S,B)\longrightarrow\CC\qquad v\longmapsto\left\langle\exp\left(\frac{\ell}{2}+\left(\frac{1}{4}+m\sqrt{-1}\right)f^*h+B\right),v\right\rangle,
\]
where $\langle-,-\rangle$ denotes the Mukai pairing and $m\in\RR_{>0}$.
Explicitly, if $v=(r,c,d)$, we have
\begin{equation}\label{eqn:stabfun}
Z_m(v)=\left( -d+U\cdot c -rU^2/2 + rm^2\right) + m\sqrt{-1}\left(c-rU \right)\cdot f^*h,
\end{equation}
where $U:=\ell/2+f^*h/4+B$.
If $r\neq0$, we can also write the real part as
\[
\Re Z_m(v) = \frac 1{2r}\left( \langle v,v\rangle + 2m^2r^2 - \left(c-rU \right)^2\right).
\]

\begin{remark}\label{rmk:realimstab}
Notice that, for any line $l\subseteq Y$, we have
\begin{align*}
& \mathrm{Re}(Z_m(v^B(\kj_l)))=\mathrm{Re}(Z_m((0,f^*h,s)))=0\\
& \mathrm{Im}(Z_m(v^B(\kj_l)))=2m.
\end{align*}
Indeed, by applying the equivalence $\Xi^{-1}$ to the triangle in Proposition \ref{prop:ideallines}, we get that $\kj_l$ is an extension of $\ke_0[1]$ and $\ke_1$:
\begin{equation}\label{eq:Columbus30}
\ke_0[1] \longrightarrow \kj_l \longrightarrow \ke_1.
\end{equation}
By Remark \ref{rmk:spherical}, $\e$ and $\f$ are both spherical, namely, for $j=0,1$, we have $\langle v^B(\ke_j),v^B(\ke_j)\rangle=-2$.
Therefore, for $j=0,1$,
\begin{align*}
& \mathrm{Re}(Z_{m}(v^B(\ke_j)))=\frac{1}{4}\left(8m^2-\frac{5}{2}\right)\\
& \mathrm{Im}(Z_m(v^B(\ke_j)))=(-1)^{j+1}m.
\end{align*}
Let us spell out this explicit calculation for $j=0$, as the case of $\ke_1$ is completely analogous. By Lemma \ref{lem:calc}, we have
\[
v^B(\ke_0)=(2,\ell+2B,s_0).
\]
Thus set $c=\ell+2B$, $r=2$ and, as above, $U=\frac{1}{2}\ell+\frac{1}{4}f^*h+B$. Hence, we have $(c-rU)=\ell+2B-2\left(\frac{1}{2}\ell+\frac{1}{4}f^*h+B\right) =-\frac{1}{2}f^*h$. In particular, we have $(c-rU)^2=\frac{1}{2}$ and $(c-rU)\cdot f^*h=-1$. Plugging these in \eqref{eqn:stabfun}, we get
\begin{equation*}
\begin{split}
Z_m(v^B(\ke_0))&=\frac 1{2r}\left( \langle v^B(\ke_0),v^B(\ke_0)\rangle + 2m^2r^2 - \left(c-rU \right)^2\right) + m\sqrt{-1}\left(c-rU \right)\cdot f^*h\\
&=\frac 14 \left(-2+8m^2 -\frac 12 \right) - m\sqrt{-1} = \frac 14\left( 8m^2 -\frac 52\right) -  m\sqrt{-1}.
\end{split}
\end{equation*}

We also observe that, for $m_0:=\frac{\sqrt{5}}{4}$, $\mathrm{Re}(Z_{m_0}(v^B(\ke_j)))=0$.
\end{remark}

To define an abelian category which is the heart of a bounded $t$-structure on $\Db(S,\beta)$, fix the rational number
\[
\mu:=\left(B+\frac{\ell}{2}+\frac{1}{4}f^*h\right)\cdot f^*h.
\]
By the previous remark, we have that $\mu=s$, where $s$ is the degree-$4$ part of $v^B(\kj_l)$.

Let $\cat{T},\cat{F}\subseteq\coh(S,\beta)$ be the following two full additive
subcategories: The non-trivial objects in $\cat{T}$ are the twisted
sheaves $\ka$ such that their torsion-free part have Harder--Narasimhan factors (with respect to $\mu^B$-stability) of slope $\mu^B>\mu$.
A non-trivial twisted sheaf $\ka$ is an object in $\cat{F}$ if $\ka$ is torsion-free and every $\mu^B$-semistable Harder--Narasimhan factor of $\ka$  has slope $\mu^B\leq\mu$.
It is easy to see that $(\cat{T},\cat{F})$ is a torsion theory (see \cite{HRS} for more details).

Following \cite{Br1}, we define the heart of the induced $t$-structure as the abelian category
\[
\cat{A}:=\left\{\ka\in\Db(S,\beta):\begin{array}{l}
\bullet\;\;\kh^i(\ka)=0\mbox{ for }i\not\in\{-1,0\},\\\bullet\;\;
\kh^{-1}(\ka)\in\cat{F},\\\bullet\;\;\kh^0(\ka)\in\cat{T}\end{array}\right\}.
\]

\begin{remark}\label{rmk:objstab}
A direct computation shows that
\[
\mu^B(\e)=\frac{(\ell+2B)\cdot f^*h}{2}<\mu<\frac{(\ell+f^*h+2B)\cdot f^*h}{2}=\mu^B(\f).
\]
Being $\e$ and $\f$ $\mu^B$-stable (Lemma \ref{prop:stab}), we have $\e[1],\f\in\cat{A}$.
Hence, by \eqref{eq:Columbus30}, $\kj_l\in\cat{A}$, when $l\subseteq P$.
On the other side, if $l\not\subseteq P$, then $\kj_l$ is in $\cat{T}\subseteq\cat{A}$, being a torsion sheaf.
Summing up, $\kj_l\in\cat{A}$, for all lines $l\subseteq Y$.
\end{remark}


\begin{lem}\label{lem:exstab}
The pair $\sigma_m:=(Z_m,\cat{A})$ defines a locally-finite stability condition for all $m>1/2$.
\end{lem}

\begin{proof}
Let $\ka\in\coh(S,\beta)$ be a torsion-free $\mu^B$-stable sheaf with $\mu^B(\ka)=\mu$, so that $$\mathrm{Im}(Z_m(v^B(\ka)))=m\left(c-r\left(\frac{\ell}{2}+\frac{1}{4}f^*h+B\right)\right)\cdot f^*h=0,$$ where $v^B(\ka)=(r,c,d)$.

To prove that the pair $\sigma_m$ defines a stability condition, we have now to analyze more in detail the real part of $Z_m(v^B(\ka))$. As $r>0$, by Lemma \ref{lem:NS}, we have $F:=\left(c-r\left(\frac{1}{2}\ell+\frac{1}{4}f^*h+B\right)\right)\in\Pic(S)\otimes\QQ$. Thus, by the Hodge Index theorem, $F\cdot F\leq 0$. Since $\ka$ is $\mu^B$-stable, we have $\chi(\ka,\ka)\leq 2$. By definition,
\[
\mathrm{Re}(Z_m(v^B(\ka)))=\frac{1}{2r}(-\chi(\ka,\ka)+2r^2m^2-F\cdot F).
\]
If $\ka$ is not spherical (which implies, since $\chi$ is even, that
$\chi(\ka,\ka)\leq0$), then $\mathrm{Re}(Z_m(v^B(\ka)))>0$ and so
$Z_m(v^B(\ka[1]))\in\HH$. If $\ka$ is spherical, then
$-2+2r^2m^2-F\cdot F>0$, because, by Corollary \ref{cor:rankdiv},
$r\geq 2$, and by assumption $m>1/2$.

As in \cite[Lemma 6.2]{Br1}, this suffices to show that
$Z_m(\ka)\in\HH$, for any $\ka\in\cat{A}$ and $m>1/2$.
The existence of the Harder--Narasimhan filtrations (which is what we need to conclude that the pair $(Z_m,\cat{A})$ is a locally-finite stability condition) follows now from the same argument as in \cite[Prop.\ 7.1 and Sect.\ 11]{Br1} adapted to the twisted setting according to \cite[Sect.\ 3.1]{HMS}.
\end{proof}

We leave it to the reader to verify that, omitting $(\ast\ast)$, the previous result does not hold anymore.

\subsection{Proof of Theorem \ref{thm:main1}}\label{subsec:proofthm2}

We keep the same assumptions as in the previous section.

\begin{lem}\label{lem:stabB01}
	The objects $\e[1]$ and $\f$ are $\sigma_m$-stable for any $m>1/2$.
\end{lem}

\begin{proof}
Suppose $\f$ is not $\sigma_m$-stable, for some $m>1/2$.
Then there exists a short exact sequence in $\cat{A}$
\begin{equation}\label{eqn:sta1}
0\longrightarrow\ka\longrightarrow\f\longrightarrow\kb\longrightarrow 0,
\end{equation}
where $\ka\neq 0$ is $\sigma_m$-semistable and it has phase $\phi_m(\ka)\geq\phi_m(\f)$.
It is not hard to see that $\ka$ is a sheaf.
By  Remark \ref{rmk:realimstab}, $\mathrm{Im}(Z_m(v^B(\f)))=m$ and, by Corollary \ref{cor:rankdiv}, $\frac{1}{m}\mathrm{Im}(Z_m(v^B(-)))$ is an integral function which is additive on triangles.
Hence, $\mathrm{Im}(Z_m(v^B(\ka)))$ is either $0$ or $m$.
The former is impossible because $\f$ is locally free and $\ka$ would then be supported on points.
If the latter holds true, then $\mathrm{Im}(Z_m(v^B(\kb)))=0$, and so \eqref{eqn:sta1} does not destabilizes $\f$.

The same argument works for $\e[1]$, by using Lemma \ref{prop:stab}.
\end{proof}

Fix a (non-canonical) isomorphism $u:\Hom(\f,\e[2])\isomor\Hom(\e,\f)$.
For $l$ a line in $Y$, consider the following sheaves in $\coh(S,\beta)$
\[
\kk_l:=\left\{\begin{array}{ll}
\kj_l&\mbox{if } l\not\subseteq P\\
\mathrm{Coker}(s_l:\e\to\f) &\mbox{otherwise,}
\end{array}\right.
\]
where $s_l=u(\overline{s}_l)$ and $\overline{s}_l:\ke_1\to\ke_0[2]$ is the morphism associated to the line $l\subseteq P$, according to Proposition \ref{prop:ideallines}.
Notice that the choice of the isomorphism $u$ is of no importance for our construction.
Moreover, by Lemma \ref{prop:stab}, all morphisms $\e\to\f$ are injective.

\begin{lem}\label{lem:st0}
For any line $l\subseteq Y$, the sheaf $\kk_l$ is pure of dimension $1$.
\end{lem}

\begin{proof}
If $l\not\subseteq P$, this was already observed in Section \ref{subsec:twistedK3}.
If $l\subseteq P$, then $\kk_l$ has a locally free resolution of length $1$ and so it cannot have torsion supported on points.
\end{proof}

The following is the first step in the proof of Theorem \ref{thm:main1}.

\begin{lem}\label{lem:st1}
{\rm (i)} If $m>m_0=\frac{\sqrt{5}}{4}$, then $\kk_l$ is $\sigma_m$-stable, for any line $l\subseteq Y$.

{\rm (ii)} Assume $m=m_0$. If $l\not\subseteq P$ is a line in $Y$, then $\kk_l$ is $\sigma_m$-stable. If $l\subseteq P$, then $\kk_l$ is $\sigma_{m_0}$-semistable and
\begin{equation}\label{eqn:sta4}
\f\longrightarrow\kk_l\longrightarrow\e[1]
\end{equation}
is the Jordan--H\"older filtration of $\kk_l$.
\end{lem}

\begin{proof}
Suppose that there is a line $l\subseteq Y$ and $m\geq m_0$ such that $\kk_l$ is not $\sigma_m$-stable. This means that there is a destabilizing short exact sequence in $\cat{A}$
\begin{equation}\label{eqn:sta2}
0\longrightarrow\ka\longrightarrow\kk_l\longrightarrow\kb\longrightarrow 0,
\end{equation}
with $\ka\neq 0$ and $\sigma_m$-stable (and so $\chi(\ka,\ka)\leq 2$).
Also in this case, this implies that $\ka\in\coh(S,\beta)$.
Set $v^B(\ka)=(r,c,d)$.

By Remark \ref{rmk:realimstab} we have $\mathrm{Im}(Z_m(v^B(\kj_l)))=2m$ and so $J:=\mathrm{Im}(Z_m(v^B(\ka)))\in\{0,m,2m\}$.
If $J=0$, then $\ka$ is a torsion sheaf supported on points, and this contradicts Lemma \ref{lem:st0}.
If $J=2m$, then $\kb$ has phase $1$ and \eqref{eqn:sta2} would not destabilize $\kk_l$.
Hence $J=m$.
Due to $(\ast\ast)$, since $f^*h$ is indecomposable, $r>0$.
By Corollary \ref{cor:rankdiv}, $r\geq 2$.

As $\ka$ destabilizes $\kk_l$, we have
\begin{equation}\label{eqn:re}
\mathrm{Re}(Z_m(v^B(\ka)))=\frac{1}{2r}(-\chi(\ka,\ka)+2r^2m^2-F\cdot F)\leq 0,
\end{equation}
where, by Lemma \ref{lem:NS}, $$F:=\left(c-r\left(\frac{1}{2}\ell+\frac{1}{4}f^*h+B\right)\right)\in\Pic(S)\otimes\QQ.$$
Since $J=m$, the Hodge Index theorem yields $F\cdot F\leq\frac{1}{2}$. In particular, by Remark \ref{rmk:realimstab},
\begin{equation}\label{eqn:re1}
\mathrm{Re}(Z_m(v^B(\ka)))\geq\frac{1}{4}\left(8m^2-\frac{5}{2}\right)=\mathrm{Re}(Z_m(v^B(\f)))\geq 0,
\end{equation}
for $m\geq m_0$. Therefore, \eqref{eqn:re} is not verified unless $m=m_0$. This proves (i).

If $m=m_0$, the only possibility for \eqref{eqn:sta2} to destabilize $\kk_l$ is that $r=2$, $\chi(\ka,\ka)=2$ and $\mathrm{Re}(Z_{m_0}(v^B(\ka)))=\mathrm{Re}(Z_{m_0}(v^B(\f)))=0$. Under these assumptions, $$E:=c_1^B(\ka)-c_1^B(\f)=c-\ell-f^*h-2B$$ is in $\Pic(S)$ (use again Lemmas \ref{lem:NS} and \ref{lem:calc}) and $F=E+\frac{1}{2}f^*h$. As $\mathrm{Im}(Z_{m_0}(v^B(\ka)))=\mathrm{Im}(Z_{m_0}(v^B(\f)))=m_0$, we have $E\cdot f^*h=0$. Then, either $E$ is trivial, or $E\cdot E<0$. The latter cannot be true because
$$\frac{1}{2}=F\cdot F=E\cdot E+\frac{1}{2}<\frac{1}{2}.$$

If $E$ is trivial, then the fact that $\chi(\ka,\ka)=\chi(\f,\f)=2$ implies that $v^B(\ka)=v^B(\f)$ and $v^B(\kb)=v^B(\e[1])$. Thus, by Lemma \ref{lem:stabB01},  $\ka\cong\f$ and $\kb\cong\e[1]$. Therefore, by Serre duality and Proposition \ref{prop:ideallines}, \eqref{eqn:sta2} destabilizes $\kk_l$ if and only if $l\subseteq P$. In this situation,
\[
\begin{array}{c}
\mathrm{Re}(Z_{m_0}(v^B(\e[1])))=\mathrm{Re}(Z_{m_0}(v^B(\f)))=0,\\
\mathrm{Im}(Z_{m_0}(v^B(\e[1])))=\mathrm{Im}(Z_{m_0}(v^B(\f)))=m_0.
\end{array}
\]
So, due to Lemma \ref{lem:stabB01}, \eqref{eqn:sta4} is the Jordan--H\"older filtration of $\kk_l$.
\end{proof}

By the discussion in \cite[Sect.\ 9]{Br1} and by Lemma \ref{lem:st1}
(ii), for $\epsilon$ sufficiently close to $m_0$, with $1/2<\epsilon<m_0$, and for a line $l\not\subseteq P$, we have that $\kj_l(=\kk_l)$ is $\sigma_m$-stable, for all $m>\epsilon$ (use that $v^B(\kj_l)$ is primitive).

Let now $l\subseteq Y$ be a line contained in $P$. By Lemma
\ref{lem:stabB01}, the
triangle $$\e[1]\longrightarrow\kj_l\longrightarrow\f$$ is the
Harder--Narasimhan filtration of $\kj_l$ for $m>m_0$. This means that
$\kj_l$ is $\sigma_{m_0}$-semistable with Jordan--H\"older factors
$\e[1]$ and $\f$. As a consequence, up to choosing $\epsilon$ closer
to $m_0$, $\kj_l$ is $\sigma_m$-stable, for all  $m\in(\epsilon,m_0)$. Indeed, if not, by \cite[Prop.\ 9.3]{Br1}, the Harder--Narasimhan factors of $\kj_l$ in the stability condition $\sigma_m$, for $m\in(\epsilon,m_0)$, would survive in the stability condition $\sigma_{m_0}$. This would contradict the $\sigma_{m_0}$-semistability of $\kj_l$.

\smallskip

To finish the proof of Theorem \ref{thm:main1}, first of all we observe that all the arguments in \cite{Toda} generalize to the twisted setting.
In particular, for all $m>\epsilon$, it makes sense to speak about the moduli space $M^{\sigma_m}(S,v,B)$ of $\sigma_m$-stable objects in $\cat{A}$ with Mukai vector $v:=v^B(\kj_l)$ as algebraic space over $\CC$. Moreover, by the results in \cite[Sect.\ 3]{I}, $M^{\sigma_m}(S,v,B)$ is smooth symplectic of dimension $4$.
Hence, the only thing we need to prove is that, for $m\in(\epsilon,m_0)$, the objects $\kj_l$ are the only $\sigma_m$-semistable objects in $\cat{A}$ with Mukai vector $v$.

Let $\kg\in\cat{A}$ be a $\sigma_m$-semistable object, for some $m\in(\epsilon,m_0)$, with $v(\kg)=v$.
By \cite[Prop.\ 9.3]{Br1}, up to replacing $\epsilon$, we can assume that all such objects $\kg$ are $\sigma_{m_0}$-semistable.

\begin{lem}\label{lem:aux}
Let $m_1\geq m_0$ and let $\kg$ be a $\sigma_{m_1}$-stable object with Mukai vector $v$.
Then $\kg$ is $\sigma_m$-stable, for all $m\geq m_1$.
\end{lem}

\begin{proof}
We use a similar argument as in the proof of Lemma \ref{lem:st1}.
Assume, for a contradiction, that $\kg$ is properly semistable at $m> m_1$ (i.e., it is not stable).
Then we have an exact sequence in $\cat{A}$
\[
0\longrightarrow\ka\longrightarrow\kg\longrightarrow\kb\longrightarrow 0,
\]
where $\ka\neq0$ is $\sigma_{m}$-stable and $\mathrm{Re}(Z_m(v^B(\ka)))=0$.
Let $v^B(\ka)=(r,c,d)$ be the Mukai vector of $\ka$.
First of all, we observe that $r\neq 0$.
Indeed, if $r=0$, then
$\mathrm{Re}(Z_{m_1}(v^B(\ka)))=\mathrm{Re}(Z_m(v^B(\ka)))=0$,
which would contradict the stability of $\kg$ at $m_1$.

As before, if we let $J:=\mathrm{Im}(Z_m(v^B(\ka)))$, then $J\in\{0,m,2m\}$.
The case $J=2m$ does not destabilize.
Since $\kg$ is $\sigma_{m_1}$-stable, then the case $J=0$ is not possible.
Hence, we are left with $J=m$.
But then, the same argument as in Lemma \ref{lem:st1}, given that $r\neq 0$ and $m>m_1$, shows that $\mathrm{Re}(Z_m(v^B(\ka)))\neq0$, which is again a contradiction.
\end{proof}

Proceeding with the proof of Theorem \ref{thm:main1}, we have two possibilities for $\kg$.
Either $\kg$ is $\sigma_{m_0}$-stable, or it is properly $\sigma_{m_0}$-semistable.
If $\kg$ is $\sigma_{m_0}$-stable, then, by Lemma \ref{lem:aux}, $\kg$ is $\sigma_m$-stable, for all $m\geq m_0$.

\begin{lem}\label{lem:aux2}
Let $\kg\in\cat{A}$ be a $\sigma_m$-stable object, for all $m\geq m_0$, with Mukai vector $v$.
Then $\kg$ is a $\beta$-twisted stable sheaf, pure of dimension $1$.
\end{lem}

\begin{proof}
The proof is very similar to \cite[Prop.\ 14.2]{Br1}.
Assume first that $\kg$ is not a sheaf, namely $\kh^{-1}(\kg)\neq0$.
Since $r(\kg)=0$, then $\kh^0(\kg)\neq0$.
Also, by definition, $\mathrm{Re}(Z_m(v^B(\kh^{-1}(\kg))))<0$, for $m\gg0$.
Hence, the exact sequence in $\cat{A}$
\[
0\longrightarrow\kh^{-1}(\kg)[1]\longrightarrow\kg\longrightarrow\kh^{0}(\kg)\longrightarrow0
\]
would destabilize $\kg$, when $m\gg0$, a contradiction.

We deduce that $\kg\cong\kh^0(\kg)$ is a sheaf.
Assume then that $\kg$ is not pure of dimension $1$.
Then its torsion part $\kg_0$ of dimension $0$ would have $Z_m(\kg_0)\in\RR_{<0}$, and would destabilize $\kg$, a contradiction.
Hence, $\kg$ is a pure $\beta$-twisted sheaf of dimension $1$, and $c_1^B(\kg)=f^*h$.
Since $f^*h$ is indecomposable, $\kg$ is stable, as we wanted.
\end{proof}

Thus, if $\kg$ is $\sigma_{m_0}$-stable, then $\kg$ is an element of $M(S,v,B)$, by Lemma \ref{lem:aux2}.
By Proposition \ref{prop:bir1}, since the moduli space $M(S,v,B)$ is irreducible, $\kg$ is then isomorphic to $\kk_l$, for some $l\subseteq Y$.
Since $\kg$ is $\sigma_{m_0}$-stable, we deduce that $\kg\cong\kj_l$.

Assume now that $\kg$ is properly $\sigma_{m_0}$-semistable.
Since $\Im(Z_{m_0}(\kg))=2m_0$, $\kg$ must have two $\sigma_{m_0}$-stable factors $\ka_0$ and $\ka_1$ with
\[
0\longrightarrow\ka_0\longrightarrow\kg\longrightarrow\ka_1\longrightarrow 0
\]
an exact sequence in $\cat{A}$ and $\Im(Z_{m_0}(\ka_0))=\Im(Z_{m_0}(\ka_1))=m_0$.
By the same argument as above, we have an extension
\[
0\longrightarrow\ka_1\longrightarrow\widetilde{\kg}\longrightarrow\ka_0\longrightarrow 0
\]
which is $\sigma_{m_1}$-stable, for some $m_1>m_0$.
By Lemma \ref{lem:aux}, $\widetilde{\kg}$ is isomorphic to $\kk_l$, for some $l\subseteq Y$ and is properly $\sigma_{m_0}$-semistable.
By Lemma \ref{lem:st1}, $\widetilde{\kg}$ is then isomorphic to the cokernel of a map $\e\to\f$.
In particular, $\ka_0\cong\e[1]$, $\ka_1\cong\f$ and so $\kg$ is isomorphic to $\kj_l$, for some $l\subseteq Y$, as wanted.
This completes the proof of Theorem \ref{thm:main1}.

\begin{remark}\label{rmk:MukaiFlopUtah}
(i) As in \cite{AB}, we can be more precise about the relation between $M^{\sigma_m}(S,v,B)\cong F(Y)$, for $m\in(\epsilon,m_0)$, and $M(S,v,B)$.
Indeed, we can show that $F(Y)$ \emph{is} the Mukai flop of $M(S,v,B)$ in the plane $\PP(\Hom(\e,\f))$.
More precisely, if $\widetilde{M}$ denotes the Mukai flop of $M(S,v,B)$ in the plane $\PP(\Hom(\e,\f))$, then, by mimicking \cite[Sect.\ 5]{AB}, we can construct a universal family $\widetilde{U}$ on $\widetilde{M}$ so that the pair $(\widetilde{M},\widetilde{U})$ represents the functor
parametrizing $\sigma_m$-stable objects in $\cat{A}$ with Mukai vector $v$.
Hence, by Theorem \ref{thm:main1}, $\widetilde{M}$ is isomorphic to $M^{\sigma_m}(S,v,B)$ and so to $F(Y)$.

(ii) Questions related to the geometry of the birational models of $F(Y)$ are addressed in \cite{HVV}. See also \cite{HT} for the study of the birational automorphisms of the Fano varieties of lines of some special cubic fourfolds.
\end{remark}


\bigskip

{\small\noindent {\bf Acknowledgements.} It is a great pleasure to thank Aaron Bertram, Bert van Geemen, and Daniel Huybrechts for very interesting and useful discussions and suggestions. The authors are grateful to Brendan Hassett, Alexander Kuznetsov and Yuri Tschinkel for comments on an early version of this article. They also thank one of the referees for finding a mistake in a previous version of the paper, and for many insightful comments and suggestions which greatly improved the exposition of the paper. This paper was concluded during the authors' stay at the Mathematische Institut of the University of Bonn whose warm hospitality is gratefully acknowledged. During his visit to Bonn, P.~S.~ was supported by the grant SFB/TR 45. E.~M.~ is partially supported by NSF grant DMS-1001482 and DMS-1160466, the Hausdorff Center for Mathematics, Bonn, and SFB/TR 45.}


\end{document}